\documentclass[preprint,10pt]{elsarticle} 
\usepackage{fullpage,amsmath,amssymb,soul}
\usepackage{geometry}

\usepackage{mathrsfs}
\usepackage[colorlinks,citecolor=red]{hyperref}
\usepackage[linesnumbered, vlined,noline, ruled]{algorithm2e}
\usepackage{enumerate}
\usepackage{pdfpages}

\usepackage{float}

\usepackage{setspace}
\usepackage{placeins}
\usepackage{soul,tikz,amsfonts, graphicx}
\usepackage{booktabs}
\usepackage{caption}
\usepackage{bm}
\usepackage{cases}
\usepackage{subfigure}

\usepackage{enumitem}

\usetikzlibrary{shapes,arrows,shapes.multipart}

\usepackage{amsthm}
\newtheorem{theorem}{Theorem}[section]
\newtheorem{remark}{Remark}[section]

\newtheorem{lemma}{Lemma}[section]

\newcommand{\HH}{\mathbb{H}}
\newcommand{\LL}{\mathcal{L}}

\graphicspath{{figs/}}

\begin{document}

\begin{frontmatter}
  \title{Isogeometric analysis for the Helmholtz transmission eigenvalue problem} 
  \author[address1]{Nizheng Liao}\ead{yc37435@um.edu.mo}
 \author[address2,address3]{Xucheng Meng\corref{cor2}}\ead{xcmeng@bnu.edu.cn}
  \cortext[cor2]{Corresponding author}
  \author[address4]{Guanghui Hu}\ead{garyhu@um.edu.mo}
\address[address1]{Department of Mathematics, University of Macau, Macao S.A.R., China} 
\address[address2]{Department of Mathematics, Faculty of Arts and Sciences, Beijing Normal University, Zhuhai 519087, Guangdong, China}
\address[address3]{Guangdong Provincial Key Laboratory of Interdisciplinary Research and Application for Data Science,  Beijing Normal-Hong Kong Baptist University, Zhuhai 519087, China}
\address[address4]{Department of Mathematics \& Guangdong-Hong Kong-Macao Joint Laboratory for Data-Driven Fluid Mechanics and Engineering Applications, University of Macau, Macao S.A.R., China}

\begin{abstract}
The transmission eigenvalue problem plays an increasingly important role in inverse scattering theory. Although significant progress has been made in developing efficient numerical methods for the problem over the past two decades, its numerical treatment for curved domains in $\mathbb{R}^d$ ($d=2,3$) remains challenging. In this paper, we introduce and analyze a geometrically flexible and $H^2$-conforming isogeometric method for solving a fourth-order, quadratic and non-self-adjoint eigenvalue problem arising from the Helmholtz transmission eigenvalue problem. Using the spectral approximation theory for compact non-self-adjoint operators, 
we derive the optimal error estimate for the discrete eigenvalues and eigenfunctions. Numerical results for both the two- and three-dimensional benchmark problems, including those defined in curved domains, are presented to verify the theoretical results, and to demonstrate the advantages of the method over existing numerical methods in terms of both accuracy and geometric flexibility.
\end{abstract}

\end{frontmatter}

\section{Introduction}

% [Significance of transmission eigenvalue problem]

% [Existing works and their drawbacks]

% [Highlights of IGA method]

% [Novelty of IGA method]

Since the transmission eigenvalues can be determined from scattering data and used to estimate the physical properties of scattering objects \cite{CAKONI2010379,Cakoni_2007, Giorgi_2012, Sun_2011_a}, the transmission eigenvalue problem, initially introduced by Colton and Monk \cite{COLTON_MONK1988} and Kirsch \cite{KIRSCH1986}, plays an increasingly important role in inverse scattering theory for inhomogeneous media \cite{Cakoni2022,CakoniHaddar2013}. The transmission eigenvalue problem has a variety of applications in inverse problems,  such as target identification, nondestructive testing \cite{Cakoni_2008,CakoniHaddar2013}, and design of invisible material \cite{JI_LIU_2018}. The problem is a non-self-adjoint eigenvalue problem that is not covered by the standard theory of elliptic eigenvalue problems \cite{CakoniHaddar2013}, making it difficult to study analytically. Therefore, the numerical method is an indispensable tool to compute the transmission eigenvalues.

% In this paper, we consider the 

% Results on the existence of transmission
% eigenvalues and related applications can be found in
% \cite{RynneSleeman1991,ColtonPaivarintaSylvester2007,PaivarintaSylvester2008,CakoniGintidesHaddar2010,CakoniHaddar2013}.

The first numerical study of the transmission eigenvalue problem  was initiated by Colton et al. \cite{colton2010analytical} in 2010, where three finite element methods were proposed. Since then,
the development of robust and efficient numerical methods for the  problem has received considerable research interest, see, for instance, \cite{Sun2011_SIAM,cakoni2014error,ji2014multigrid,yang2016mixed,HAN201796,Mora2018,XiJiZhang_2020,mora2021virtual,MengMeiM3AS_2022}.
% and various numerical methods have been proposed for solving it. 
To facilitate the theoretical analysis and numerical treatment, the problem is usually reformulated into an equivalent  fourth-order quadratic eigenvalue problem \cite{Rynne_1991_SIAM,colton2010analytical}. Different numerical methods have been proposed to solve the transmission eigenvalue problem based on this fourth-order eigenvalue problem. For example,  the $H^2$-conforming finite element method (FEM) using either Argyris element or Bogner–Fox–Schmit (BFS) element has been considered, see, e.g.,  \cite{Sun2011_SIAM,cakoni2014error,ji2014multigrid,HAN201796}. However, the construction of $C^1$ finite elements is difficult in general, especially for the three-dimensional problems \cite{ciarlet2002finite,brenner08,HU_C_r_2023}. As alternatives to the $C^1$ finite elements, mixed FEM \cite{colton2010analytical,Ji_2012,yang2016mixed,XiJiZhang_2020}, and some nonstandard finite element methods, such as non-conforming FEM \cite{yang2016non,ji2017nonconforming,xi2020high,YANG2020112697}, discontinuous Galerkin method \cite{geng2016c,meng2023discontinuous,wang2023mixed_DG}, and a $C^0$ linear FEM using gradient
recovery operator \cite{chen2017linear},  have been successfully used to compute the transmission eigenvalues. Moreover, the $C^1$ virtual element method \cite{Mora2018,mora2021virtual,MengMeiM3AS_2022}, the mixed virtual element method \cite{meng2023mixed}, and the  spectral methods \cite{an2013spectral,an2023novel,Tan_An_2019,TanCao_2025} have also  been developed to solve the  problem.        
% Sun \cite{Sun2011_SIAM} proposed two iterative methods (bisection and secant) to compute the transmission eigenvalues, where the $H^2$ conforming finite element method (FEM) with Argyris element is used for discretization. Cakoni et al. \cite{cakoni2014error} also considered the use of $C^1$ FEM with Argyris element, and proved the error estimates for the eigenvalues using the Osborn’s perturbation theory for the compact non-self-adjoint operators. To improve   

Despite the excellent performance of the aforementioned numerical methods for the transmission eigenvalue problems in polytopal domains of~$\mathbb{R}^d$  ($d=2,3$), their accuracy may be degraded in domains with curved boundaries, see, e.g., \cite{cakoni2014error,yang2016mixed}. For some special curved domains, such as disks, spherical and cylindrical domains, the spectral methods have been studied to achieve spectral convergence, see, e.g. \cite{an2023novel,Tan_An_2019,TanCao_2025}. However, the extension to general curved domains remains nontrivial.

% In this paper, we introduce and analyze  isogeometric analysis (IGA) \cite{hughes2005isogeometric,IGA_book} to solve the fourth-order transmission eigenvalue problem in \textit{complex domains} of $\mathbb{R}^d$ ($d=2,3$). 
% We note that isogeometric analysis (IGA), proposed by Hughes et al.  \cite{hughes2005isogeometric} with the aim to bridge the gap between finite element analysis and computer-aided design (CAD), is geometrically flexible for curved domains. 

We note that isogeometric analysis (IGA), introduced by Hughes et al. \cite{hughes2005isogeometric} with the aim of bridging the gap between finite element analysis and computer-aided design (CAD), provides geometric flexibility for curved domains. As a generalization of FEM, IGA employs the spline basis functions (e.g., Non-Uniform Rational B-Splines, NURBS) that describe the CAD geometry to  construct the finite-dimensional approximation space in the Galerkin method. IGA possesses several distinguished advantages over the classical FEM: (1) it is  easy to construct globally $C^k$-continuous ($k\ge 1$) basis functions, (2) it  uses the exact geometry of CAD,  and (3) it is flexible for the $h$-, $p$-, and $k$-refinements. Consequently, the NURBS-based IGA has been successfully applied to a wide range of scientific and engineering problems, such as structural vibration problems \cite{COTTRELL20065257}, fluid dynamics \cite{NIELSEN20113242}, and we refer to \cite{IGA_book,NGUYEN201589} and  references therein for the details. Furthermore, IGA has also been utilized to solve high-order partial differential equations, such as the biharmonic and triharmonic problems \cite{TAGLIABUE2014277}, and the Cahn-Hilliard equation \cite{meng2025convergence,BARTEZZAGHI2015446,GOMEZ20084333}. % Therefore, IGA is well suited for solving the transmission eigenvalue problems in \emph{curved domains}.   

In this paper, we introduce and analyze the Geometry-Independent Field approximation (GIFT) scheme~\cite{Atroshchenko2018}, which is an extension of the NURBS-based IGA, for solving the fourth-order transmission eigenvalue problem on bounded domains of $\mathbb{R}^d$ ($d=2,3$) with curved boundaries. The GIFT scheme employs different spline basis functions for the geometric representation and for the solution approximation,  thereby facilitating  the computer implementation and enhancing computational efficiency.
To derive the optimal error estimates for both the real and complex transmission eigenvalues, and the corresponding eigenfunctions within the framework of the GIFT scheme,  we first use the linearization technique proposed in \cite{yang2015error} to transform the quadratic fourth-order eigenvalue problem into an equivalent linear eigenvalue problem, whose weak solution lies in $H_0^2(\Omega)\times L^2(\Omega)$.  We then make use of the capability of IGA to build highly smooth 
basis functions and develop the $H_0^2(\Omega)\times L^2(\Omega)$ conforming discrete approximation space. Then, we introduce the continuous/discrete solution operator whose spectrum is related to the solution of the continuous/discrete variational formulation for the fourth-order transmission eigenvalue problem, and derive % Following the proof for the NURBS approximation result \cite{Bazilevs2006} and using the
% Banach-space interpolation theory \cite{brenner08}, 
the error between the continuous and discrete solution operators. Finally, by employing the spectral approximation theory for compact non-self adjoint operators \cite{osborn1975spectral,BABUSKA1991641}, we derive the optimal order error estimates for the eigenvalues and eigenfunctions. We  remark that the linearization technique used here  has been adopted by several different numerical methods for solving the transmission eigenvalue problem, such as the virtual element method \cite{mora2021virtual},  the
mixed  virtual element method \cite{meng2023mixed}, the mixed FEM \cite{yang2016mixed}, and the non-conforming FEM \cite{yang2016non, YANG2020112697}.  Compared with these methods, the GIFT scheme is 
more geometrically flexible and requires fewer degrees of freedom to achieve a comparable level of accuracy. For example,
to achieve comparable accuracy in circular domain, the number of degrees of freedom needed by the GIFT scheme with degree $p=3$ is around one hundredth of that by the cubic mixed FEM used in \cite{yang2016mixed}.

The rest of this paper is organized as follows. In Section \ref{sec:Model_problem}, we derive the linear fourth-order formulation of the Helmholtz transmission eigenvalue problem, and introduce the associated solution operators. In Section \ref{sec:IGA_transmission_eigenvalues}, we review the definitions of B-splines and NURBS, and then present the GIFT scheme for the transmission eigenvalue problem. We prove the optimal error estimates for the eigenvalues and eigenfunctions in Section \ref{sec:error_estimate}. In Section \ref{sec:examples}, we present several two- and three-dimensional numerical examples to verify the theoretical results. Finally, the conclusion is drawn in Section \ref{sec-conclusion}.

\section{The linear formulation of the transmission eigenvalue problem} \label{sec:Model_problem}

In what follows, we assume that
$\Omega\subset\mathbb{R}^d$ $(d=2,3)$ is a bounded open domain with
Lipschitz continuous boundary $\partial\Omega$, and $\bm \nu$ is the unit outward normal vector to
$\partial\Omega$. For the Sobolev space $H^k(\Omega)$ ($k\ge 0$), the norm and inner product are denoted by $\| \cdot\|_{k,\Omega}$ and $(\cdot, \cdot)_k$, respectively. 
Let $X$ and $Y$ be normed spaces, the set of all bounded linear operators $T: X \to Y$ is denoted by $\mathcal{L}(X,Y)$, and the operator norm of $T$ is denoted by $\| T \|_{\mathcal{L}(X,Y)}$. We write $\mathcal{L}(X) := \mathcal{L}(X,X)$ for brevity. Moreover, for a complex number or complex-valued function $f$, we use ${\bar{f}}$ to denote its complex conjugate.

% In particular, we denote the inner product and norm over the Hilbert space $L^2(\Omega) = H^0(\Omega)$ by $(\cdot,\cdot)_0$ and $\| \cdot \|$, respectively. 

We consider the following Helmholtz transmission
eigenvalue problem \cite{CakoniHaddar2013,colton2010analytical,Mora2018,mora2021virtual}: Find the transmission eigenvalue
$k\in\mathbb{C}$ and a nontrivial pair $(u_1,u_2)\in L^2(\Omega) \times L^2(\Omega)$, with
$u_1-u_2\in H^2(\Omega)$, such that
\begin{equation}
\displaystyle
\label{eq:transmission_eigenvalues_eqs}
\begin{cases}
\Delta u_1+k^2 n u_1  & = 0   \quad \mbox{ in } \Omega, \\
\Delta u_2+k^2 u_2    & = 0   \quad \mbox{ in } \Omega,  \\
u_1-u_2  & = 0   \quad \mbox{ on }  \partial\Omega,  \\
\displaystyle \frac{\partial u_1}{\partial \bm \nu}-\frac{\partial u_2}{\partial \bm \nu}
 & =0   \quad \mbox{ on } \partial\Omega,  
\end{cases}
\end{equation}
where $n:=n(\bm{x})\in W^{2,\infty}(\Omega)$ is the index of refraction, satisfying 
\[
n(\bm x)-1 \ge \alpha  \quad \mbox{a.e.} \quad \mbox{in} \quad \Omega,
\]
for some constant $\alpha>0$. The following theoretical analysis also holds, with obvious modifications, when $n(\bm x)>0$ is
strictly less than one. 
% A nonzero wave number $k$ for which
% \eqref{eq:transmission_eigenvalues_eqs} admits a nontrivial solution is called a
% transmission eigenvalue. The resulting spectral problem is nonlinear,
% non-self-adjoint, and not covered by the standard theory for self-adjoint elliptic
% eigenvalue problems \cite{ColtonMonkSun2010,CakoniHaddar2013}. 
% For numerical
% approximation, the coupled second-order system is often reduced to a fourth-order
% equation. 

Following \cite{Rynne_1991_SIAM,colton2010analytical,ji2017nonconforming}, by setting  $u=u_1-u_2\in H_0^2(\Omega)$, where 
\[
H_0^2(\Omega) = \{ u\in H^2(\Omega), u = \frac{\partial u}{\partial \bm \nu} =0 \,\,\mbox{ on }\,\, \partial \Omega \},
\]
we can reformulate the coupled second-order system  \eqref{eq:transmission_eigenvalues_eqs}  as  a fourth-order eigenvalue problem: Find $k\in \mathbb{C}$ and a nontrivial $u$ such that
\begin{equation}\label{eq:fourth_order_eigenvalue}
\begin{cases}
\displaystyle \left(\Delta+k^2 n\right)
\left[\frac{1}{n-1}\left(\Delta+k^2\right)u\right] =0 \quad & \mbox{ in } \Omega, \\
\displaystyle u = \frac{\partial u}{\partial \bm \nu}  = 0 \quad & \mbox{ on } \partial\Omega.
\end{cases}
\end{equation}
% This fourth-order form gives a
% variational eigenvalue problem in $H_0^2(\Omega)$, but conforming discretizations must
% provide globally $C^1$ approximation spaces, and the resulting eigenvalue problem
% remains quadratic and non-self-adjoint.

% Multiplying the governing equation \eqref{eq:fourth_order_eigenvalue} by a test function $v\in H_{0}^2(\Omega)$, and applying Green's formula, we obtain the following 
The continuous variational formulation of the fourth-order eigenvalue problem \eqref{eq:fourth_order_eigenvalue} reads: Find $k\in \mathbb{C}$ and $u\in H_0^2(\Omega)$ with $u\ne 0$, such that
\begin{equation}\label{eq:fourth_eigen}
    \Big( \frac{1}{n-1}\Delta u, \Delta v \Big)_0 + k^2 \Big[ \Big(\frac{n}{n-1}  \Delta u, v\Big)_{0}    + \Big( \frac{1}{n-1}u, \Delta  v\Big)_{0} \Big] + k^4 \Big( \frac{n}{n-1}u,  v \Big)_{0}  = 0  \quad \forall v \in H_0^2(\Omega).
\end{equation}
It is easy to see that $k=0$ is not an eigenvalue of \eqref{eq:fourth_eigen}.
% where $\overline{v}$ is the complex conjugate of $v$. 

If we set $\lambda := -k^2$, then the variational problem \eqref{eq:fourth_eigen} is a quadratic eigenvalue problem with respect to $\lambda$, that is, we have
\begin{equation}\label{eq:quadratic_fourth_eigen}
    \Big( \frac{1}{n-1}\Delta u, \Delta v \Big)_0 = \lambda \Big[ \Big(\frac{n}{n-1}  \Delta u, v\Big)_{0}    + \Big( \frac{1}{n-1}u, \Delta  v\Big)_{0} \Big] - \lambda^2 \Big( \frac{n}{n-1}u,  v \Big)_{0}    \qquad \forall v \in H_0^2(\Omega).
\end{equation}

To linearize the quadratic eigenvalue problem \eqref{eq:quadratic_fourth_eigen}, we utilize the linearization technique proposed in \cite{yang2015error}. % to transform \eqref{eq:quadratic_fourth_eigen} into a linear eigenvalue problem. 
To this end, we introduce an auxiliary variable
\begin{equation}\label{eq:z_k_2_u}
    % z := k^2 u = -\lambda u~~ \mbox{in} ~~ \Omega,
    z :=  -\lambda u \qquad \mbox{in} ~~ \Omega,
\end{equation}
and consequently, we have
\begin{equation}\label{eq:z_u_weak_form}
    (z,w)_{0} = -\lambda (u,w)_{0}~~ \quad \forall w\in L^2(\Omega). 
\end{equation}
We define  the product space $\mathbb{H} := H_0^2(\Omega)\times L^2(\Omega)$, endowed with the product norm
\[
\| (u,z) \|_{\mathbb{H}} := \big(\, \|u\|_{2,\Omega}^2 + \|z\|_{0, \Omega}^2 \big)^{1/2}   \qquad \forall (u,z)\in \mathbb{H}.
\]
Combining \eqref{eq:quadratic_fourth_eigen}, \eqref{eq:z_k_2_u}, and \eqref{eq:z_u_weak_form}, we obtain the following linear eigenvalue formulation: Find $\big( \lambda, (u,z) \big)\in \mathbb{C}\times \mathbb{H}$ with $\lambda \ne 0$ and $u\ne 0$ such that 
\begin{equation}\label{eq:linear_eigenvalue_problem}
    A\big( (u,z),(v,w)   \big) = \lambda B( (u,z),(v,w)  )  ~~\quad \forall (v,w)\in \mathbb{H}, 
\end{equation}
where $A(\cdot, \cdot ): \mathbb{H}\times \mathbb{H} \to \mathbb{C}$, and $B(\cdot, \cdot ): \mathbb{H}\times \mathbb{H} \to \mathbb{C}$ are sesquilinear forms, defined by
\begin{equation}
    A\big( (u,z),(v,w)   \big) := \Big( \frac{1}{n-1}\Delta u, \Delta v \Big)_0 + (z,w)_0, 
\end{equation}
and
\begin{equation}
    B\big( (u,z),(v,w)   \big) := \Big(\frac{n}{n-1}  \Delta u, v\Big)_0    + \Big( \frac{1}{n-1}u, \Delta  v\Big)_0 + \Big( \frac{n}{n-1}z,  v \Big)_0 - (u,w)_0.
\end{equation}

% Several different numerical methods have been proposed for solving the variational formulation \eqref{eq:fourth_eigen} based on the linear eigenvalue problem \eqref{eq:linear_eigenvalue_problem}, including the $C^1$-$C^0$ conforming FEM~\cite{yang2015error}, the mixed FEM~\cite{yang2016mixed}, the $C^0$ interior penalty Galerkin ($C^0$IPG) FEM~\cite{YANG201771}, the non-conforming FEM~\cite{yang2016non,YANG2020112697}, the $C^1$-conforming virtual element method~\cite{mora2021virtual}, and the mixed virtual element method~\cite{meng2023mixed}. Although these numerical methods perform well on polytopal domains, they generally fail to deliver high-order accuracy on domains with a curved boundary without a specialized boundary approximation, see, e.g., \cite{yang2016mixed,cakoni2014error}. In this paper, we introduce and analyze a geometrically flexible and $H^2$-conforming isogeometric discretization for the linear eigenvalue problem \eqref{eq:linear_eigenvalue_problem}. 

Regarding the properties of the sesquilinear forms $A(\cdot, \cdot)$ and $B(\cdot, \cdot)$, we have the following lemma, and we refer to  \cite{yang2015error,mora2021virtual} for the details.

\begin{lemma}\label{lemma_A_B} 
There exist positive constants $\beta$ and $C$ depending on the index of
refraction $n$ such that
\begin{equation}
\begin{aligned}
&A\big( (u,z), (u,z) \big)\ge \beta \, \| (u,z) \|_{\mathbb{H}}^2, \\
&|A\big( (u,z), (v,w) \big)|\le C \| (u,z) \|_{\mathbb{H}} \| (v,w) \|_{\mathbb{H}},\\
&|B\big( (u,z), (v,w) \big)|\le C \| (u,z) \|_{\mathbb{H}} \| (v,w) \|_{\mathbb{H}},
    \end{aligned}
\end{equation}
for all $(u,z)$,   $(v,w)\in \mathbb{H}$. Moreover, we have $ A\big( (u,z), (v,w) \big) = \overline{A\big( (v,w), (u,z) \big)}$.
\end{lemma}

\begin{remark}
It follows from Lemma \ref{lemma_A_B} that the sesquilinear form $A(\cdot,\cdot)$ defines an inner product on $\mathbb{H}$.
\end{remark}

% which is equivalent to the standard

The source problem associated with the linear eigenvalue problem \eqref{eq:linear_eigenvalue_problem} reads: For any given $(f,g)\in \mathbb{H}$, find $(\widetilde{f}, \widetilde{g})\in \mathbb{H}$ such that 
\begin{equation}\label{eq:source_problem}
 A\big( (\widetilde{f},\widetilde{g}),(v,w)   \big) =  B\big( (f,g),(v,w)  \big)  ~~ \quad \forall (v,w)\in \mathbb{H}.    
\end{equation}

According to Lemma \ref{lemma_A_B} and the Lax-Milgram theorem, the source problem \eqref{eq:source_problem} admits a unique solution $(\widetilde{f},\widetilde{g})\in \mathbb{H}$. Therefore, for any given $(f,g)\in  \mathbb{H} $, we define the solution operator $T: \mathbb{H} \to \mathbb{H}$ by
\begin{equation}
    T(f,g) = (\widetilde{f}, \widetilde{g}),
\end{equation}
where $(\widetilde{f}, \widetilde{g})$ is the unique solution to the source problem \eqref{eq:source_problem}. 
It is straightforward  to verify that the solution operator $T$ is  well-defined, linear, and bounded.  Note that $\big(\lambda, (u,z) \big)$ is an eigenpair of \eqref{eq:linear_eigenvalue_problem} if and only if $\big(\lambda^{-1}, (u,z)\big)$ is an eigenpair of $T$, that is, 
\begin{equation}\label{eq:operator_T}
  T(u,z) = \mu (u,z), \quad \mbox{where} \quad \mu := \lambda^{-1},   
\end{equation}
see \cite{yang2015error,mora2021virtual} for the details.  We remark that no spurious eigenvalues are introduced into the problem \eqref{eq:operator_T} since if $\mu \ne 0$, then $(0,z)$ is not an eigenfuntion of the problem \eqref{eq:operator_T}.

The regularity result for the solution to the
source problem \eqref{eq:source_problem} is stated
in the following lemma, see also Lemma 2.2 of \cite{mora2021virtual}.

\begin{lemma}\label{lemma:regularity_result}
There exist a real number $\sigma\in (0, 1]$ and a positive constant $C$ depending on the index of refraction $n$ such that for all $(f, g) \in \mathbb{H}$, the solution of \eqref{eq:source_problem} satisfies  $(\widetilde{f},  \widetilde{g} ) \in H^{2+\sigma}(\Omega ) \times H_0^
2(\Omega ) $, and
\[
\| \widetilde{f} \|_{2+\sigma, \Omega} + \| \widetilde{g} \|_{2, \Omega} \le C \|  (f,g) \|_{\mathbb{H}}.
\]
\end{lemma}

It follows from the compact embedding $H^{2+\sigma}(\Omega)\times H_0^2(\Omega) \hookrightarrow \mathbb{H} $ and Lemma \ref{lemma:regularity_result} that the solution operator $T: \mathbb{H}\to \mathbb{H}$ is compact.

Since the linear eigenvalue problem \eqref{eq:linear_eigenvalue_problem} is non-self-adjoint, we also need to consider its adjoint problem: Find $\big(\lambda^{*}, (u^*,z^*) \big)\in \mathbb{C}\times \mathbb{H}$ such that
\begin{equation}\label{eq:adjoint_eigenvalue_problem}
    A\big( (v,w), (u^*,z^*)  \big) = \overline{\lambda^*} \,B\big( (v,w), (u^*,z^*) \big) \qquad \forall (v,w)\in \mathbb{H},
\end{equation}
and the corresponding source problem: For any given $(f^*,g^*)\in \mathbb{H}$, find $(\widetilde{f}^*, \widetilde{g}^*)\in \mathbb{H}$ such that
\begin{equation}\label{eq:adjoint_source_problem}
 A\big( (v,w), (\widetilde{f}^*, \widetilde{g}^*)  \big) =  B\big( (v,w), (f^*,g^*) \big) \qquad \forall (v,w)\in \mathbb{H}.    
\end{equation}
By Lemma \ref{lemma_A_B} and the Lax-Milgram theorem, the adjoint problem \eqref{eq:adjoint_source_problem}  admits a unique solution for each $(f^*,g^*)\in \mathbb{H}$. We therefore define the  solution operator $T^*: \mathbb{H}\to \mathbb{H}$ by $T^*(f^*,g^*) := (\widetilde{f}^*, \widetilde{g}^*)$, which  satisfies
\begin{equation}
 A\big( (v,w), T^*(f^*,g^*) \big) =  B\big( (v,w), (f^*,g^*) \big) \quad \forall (v,w)\in \mathbb{H}.       
\end{equation}
It is straightforward to verify that the adjoint eigenvalue problem \eqref{eq:adjoint_eigenvalue_problem} admits the equivalent operator formulation
\begin{equation}
    T^*(u^*,z^*) = \frac{1}{\lambda^*}(u^*,z^*).
\end{equation}

It can be proved that $T^*$ is the adjoint operator of $T$ with respect to the inner product $A(\cdot, \cdot)$.  Consequently, the eigenvalues of  \eqref{eq:linear_eigenvalue_problem} and \eqref{eq:adjoint_eigenvalue_problem} are related by $\lambda = \overline{\lambda^*}$. Furthermore, the regularity result of the adjoint problem \eqref{eq:adjoint_source_problem} is provided in the following lemma, see Lemma 2.4 of  \cite{mora2021virtual}.

\begin{lemma}\label{lemma:adjoint_regularity_result}
There exist a real number $\sigma\in (0, 1]$ and a positive constant $C$ depending on the index of refraction $n$ such that for all $(f^*, g^*) \in \mathbb{H}$, the solution  of \eqref{eq:adjoint_source_problem} satisfies  $(\widetilde{f}^*,  \widetilde{g}^* ) \in H^{2+\sigma}(\Omega ) \times H_0^
2(\Omega ) $, and
\[
\| \widetilde{f}^* \|_{2+\sigma,\Omega} + \| \widetilde{g}^* \|_{2,\Omega} \le C \|  (f^*,g^*) \|_{\mathbb{H}}.
\]
\end{lemma}

\section{Isogeometric discretization for the transmission eigenvalue problem} \label{sec:IGA_transmission_eigenvalues}

In this section, we employ the geometry-independent field approximation (GIFT) scheme proposed in~\cite{Atroshchenko2018} to solve the linear  eigenvalue problem \eqref{eq:linear_eigenvalue_problem}. 
%For the GIFT scheme considered here, the physical domain $\Omega$ is represented by the NURBS geometric mapping $\mathbf{F}:\widehat{\Omega}\to\Omega$ defined at the coarsest level,  where $\widehat{\Omega} = (0,1)^d$ ($d=2,3$) is the parametric domain, while the finite-dimensional approximation space $V_h \subset H_0^2(\Omega)$ is constructed from B-spline basis functions defined on $\widehat{\Omega}$ and the push forward to $\Omega$ via $\mathbf{F}$. 
% A key feature of this scheme is that mesh refinement is performed exclusively on the field approximation space, leaving the geometric description unchanged at the coarsest level, thereby preserving the exact geometry while enhancing approximation capacity. 
For a detailed description of B-splines and NURBS, we refer the reader to~\cite{piegl2012nurbs}; for the fundamentals of NURBS-based IGA, see~\cite{hughes2005isogeometric, IGA_book}.  Here, we restrict our attention to the two-dimensional splines, and the extension to three dimensions is straightforward.

\subsection{B-splines and NURBS}

For the sake of completeness, we briefly present an overview of B-spline and NURBS basis functions. We first introduce the knot vector
\begin{equation}\label{eq:knot_vector_xi}
  \Xi =\{0= \xi_1,\,\xi_2,\,\dots,\,\xi_{n_{\xi}+ p +1} = 1\},
\end{equation}
where $n_{\xi}$ and $p$ denote the number and degree of the B-spline basis functions, respectively, and the knots $\xi_i \in \mathbb{R}$ satisfy 
$\xi_1 \le \xi_2 \le \cdots \le \xi_{n_{\xi}+p+1}$. 
Given the knot vector $\Xi$ \eqref{eq:knot_vector_xi}, we define the one-dimensional parametric
domain $\widehat{\Omega}:=(0,1)$. 
Knots can be repeated, and we will consider the so-called open knot vectors, % \cite{piegl2012nurbs,hughes2005isogeometric},
in which the first as well as the last knots appear $p + 1$ times. 
% In the following,
% we only consider the open knot vectors.

The $i$-th univariate B-spline basis function $N_{i,p}(\xi)$, $1\le i \le n_{\xi}$,  is a piecewise 
polynomial of degree $p$,  defined by the Cox--de Boor recursion formula
\begin{equation}\label{recursive}
  N_{i,k}(\xi)=\frac{\xi-\xi_i}{\xi_{i+k}-\xi_i}N_{i,k-1}(\xi)+\frac{\xi_{i+k+1}-\xi}{\xi_{i+k+1}-\xi_{i+1}}N_{i+1,k-1}(\xi),
\end{equation}
for $k=1,2,\ldots,p$, and 
\begin{equation}\label{zerodegree}
  N_{i,0}(\xi) = \left \{ \begin{array}{l}
    1  \qquad \mbox{ if $\xi_i\le \xi < \xi_{i+1}$},\\
    0  \qquad \mbox{ otherwise},
  \end{array}\right .
\end{equation}
where the quotient $0/0$ is set to  zero, and $\xi \in \overline{\widehat{\Omega}} = [0,1]$.

The B-spline basis functions possess several  excellent properties, including (1)  non-negativity, (2) partition of unity, (3) $N_{i,p}(\xi)$ has local support, and its support is $[\xi_i,\xi_{i+p+1}]$, and (4) $N_{i,p}(\xi)$ is $C^{p-m_j}$ at
$\xi=\xi_{j}$, where $m_j$ is the multiplicity of the knot $\xi_{j}$ in $[\xi_i,\xi_{i+p+1}]$.

Given the following two open knot vectors
\[
\Xi = \{\xi_1 = \cdots = \xi_{p+1} = 0 < \xi_{p+2} \le \cdots \le \xi_{n_\xi} < \xi_{n_\xi+1} = \cdots = \xi_{n_\xi + p + 1} = 1\},
\]
and
\[
\mathcal{H} = \{\eta_1 = \cdots = \eta_{q+1} = 0 < \eta_{q+2} \le \cdots \le \eta_{n_\eta} < \eta_{n_\eta+1} = \cdots = \eta_{n_\eta + q + 1} = 1\},
\]
where $n_\xi$ and $n_\eta$ are the numbers of univariate B-spline basis functions in the $\xi$- and $\eta$-parametric directions, respectively, and $p$ and $q$ denote the corresponding degrees, the bivariate (tensor-product) B-spline basis functions are defined as
\begin{equation}\label{eq:B-spline}
    B_{i,j}^{(p,q)}(\xi,\eta) := N_{i,p}(\xi) M_{j,q}(\eta), \qquad 1\le i \le n_{\xi}, \,\, 1 \le j \le n_\eta,
\end{equation}
where $N_{i,p}(\xi)$ and $M_{j,q}(\eta)$ are the univariate B-splines defined on the knot vectors $\Xi$ and $\mathcal{H}$, respectively. Associated with these two knot vectors is a mesh for the two-dimensional parametric domain $\widehat{\Omega} =  (0,1)^2$:
\begin{equation}\label{parameter_mesh}
    \widehat{\mathcal{T}}_h = \Big\{ \widehat{K} \;\Big|\; \widehat{K} = [\xi_i,\xi_{i+1}] \times [\eta_j,\eta_{j+1}],
    \ \text{with } \operatorname{meas}(\widehat{K}) \ne 0, \quad 1\le i \le n_\xi,\; 1\le j \le n_\eta \Big\}, 
\end{equation}
where the subscript $h$ of $ \widehat{\mathcal{T}}_h$ is the global mesh size, which is defined in \eqref{eq:mesh_size_h}. Moreover,    we let $h_{\widehat{K}}$ be the diameter of $\widehat{K} \in \widehat{\mathcal{T}}_h $. We also assume that all of the meshes are \textit{shape regular}, that is, there exists a positive constant $C$ independent of $h$ such that
\begin{equation}
\frac{  h_{\widehat{K}}   }{     \rho_{\widehat{K}} }
 \le C \qquad  \forall \widehat{K} \in \widehat{\mathcal{T}}_h,
\end{equation}
where $ \rho_{\widehat{K}} $ is the smallest edge of $\widehat{K}$.
% that is, $ h_{\widehat{K}} = \operatorname{diam}(\widehat{K})$.

The bivariate Non-Uniform Rational B-Spline (NURBS) basis functions 
% are based on the bivariate B-splines and a set of 
% positive weights, and 
are defined by
\begin{equation}
  R_{i,j}^{(p,q)}(\xi,\eta) :=\frac{w_{i,j} B_{i,j}^{(p,q)}(\xi,\eta)}{w(\xi,\eta)}, \quad \mbox{for}\quad 1\le i \le n_\xi,\, 1\le j \le n_\eta,
 \end{equation}
 where $w_{i,j}>0$ are the weights, and $w(\xi,\eta)$ is the  weight function defined by
 \begin{equation}
  w(\xi,\eta) := \sum_{i=1}^{n_\xi} \sum_{j=1}^{n_\eta} w_{i,j} B_{i,j}^{(p,q)}(\xi,\eta).      
 \end{equation}

 % Now let us introduce the bivariate NURBS space over the parametric domain $\widehat{\Omega}$
 % \begin{equation}
 %       \mathcal{N} = \emph{span} \Big \{  R_{i,j}^{(p,q)}(\xi,\eta)  \Big\}_{1\le i \le n,\, 1\le j \le m}. 
 % \end{equation}

%  \begin{remark}
% The NURBS space reduces to a B-Splines space when all weights
% $\{w_{i,j}\}_{1\le i \le n,1\le j\le m}$ are equal to a constant, 
% due to the partition of unity  of B-spline basis functions.

% \end{remark}

We suppose that the physical domain $\Omega \subset \mathbb{R}^2$ can be  parametrized by a global NURBS geometric mapping, that is, $\Omega = {\bf F}(\widehat \Omega)$, where ${\bf F}$ is given by
 \begin{equation}
 {\bf{F}}(\xi,\eta) = \sum_{i=1}^{n_\xi} \sum_{j=1}^{n_\eta} {\bf{P}}_{i,j}\, R_{i,j}^{(p,q)}(\xi,\eta), 
 \quad \mbox{for} \quad (\xi,\eta) \in \overline{\widehat{\Omega}}=[0,1]^2,    
 \end{equation}
 here, ${\bf P}_{i,j}\in \mathbb{R}^2$, $1\le i \le n_\xi$, $1\le j \le n_\eta$, are the control points.

 The physical mesh for the physical domain $\Omega$ is given by
\begin{equation}\label{physical_mesh}
\mathcal{T}_h=\Big\{ K\big| K= {\bf F}(\widehat{K}) \subset \Omega, \quad \forall \widehat{K} \in \widehat{\mathcal{T}}_h  \Big\},      
\end{equation}
where the subscript $h$ denotes the mesh size,  defined by
\begin{equation}\label{eq:mesh_size_h}
h:= \max \{h_K: K\in \mathcal{T}_h\}, 
\end{equation}
with $h_K$ being defined as (see, e.g., \cite{Bazilevs2006})
\begin{equation}
h_K= \|  \nabla {\bf F}\|_{L^{\infty}(\widehat{K})} \, h_{\widehat{K}}.   
\end{equation} 
Following \cite{Bazilevs2006,TAGLIABUE2014277}, we assume that $\bf F$ is invertible a.e. in $\Omega$, with  smooth inverse on every element ${K}\in {\mathcal{T}}_h$. 

In the framework of IGA, there are three refinement methods: 
(i) $h$-refinement: knot insertion, 
(ii) $p$-refinement: degree elevation, and 
(iii) $k$-refinement: $p$-refinement followed by $h$-refinement; 
we refer to \cite{hughes2005isogeometric,IGA_book} for further details. 
After a refinement step, we continue to denote by $n_\xi$ and $n_\eta$ 
the numbers of spline basis functions in the $\xi$- and $\eta$-directions, 
respectively. If the original degrees $p$ and $q$ are unequal, 
we may perform degree elevation on the lower-degree direction so that, 
without loss of generality, $p = q$.

\subsection{Isogeometric discretization }

We now present the GIFT scheme~\cite{Atroshchenko2018}, which generalizes the NURBS-based IGA, for discretizing  the linear eigenvalue problem \eqref{eq:linear_eigenvalue_problem}. Unlike the NURBS-based IGA, where NURBS basis functions are used to construct the finite-dimensional approximation space, in the GIFT scheme we employ the B-spline basis functions to construct the finite-dimensional approximation space $V_h$ as follows:
\begin{equation}\label{eq:S_h}
    V_h := \operatorname{span}\left\{ B_{i,j}^{(p,q)}\circ \mathbf{F}^{-1} :\ 1\le i \le n_\xi,\ 1\le j \le n_\eta \right\} \cap H_0^2(\Omega),
\end{equation}
where $p=q\ge 2$, $B_{i,j}^{(p,q)}$ are the bivariate B-splines generated from the (refined) knot vectors used to construct the NURBS geometric mapping $\bf{F}$, and $\mathbf{F}^{-1}$ is the inverse of  $\mathbf{F}$. Since the parameterization of $\mathbf{F}$ remains fixed during $h$-, $p$-, and $k$-refinements, in our GIFT scheme we only need to refine the bivariate B-splines $B_{i,j}^{(p,q)}$, while the knot vectors and control net of $\mathbf{F}$ remain as initially defined. Compared with the classical NURBS-based IGA, the GIFT scheme is easier to implement and more efficient, as there is no need to update the control points and weights of the geometric mapping in the computation.

% 

% Moreover, let \(h:=\max_{K\in\mathcal K_h}h_K\). We denote by
% \(\widetilde Q\) the support extension of \(Q\) and set
% \(\widetilde K:=\mathbf F(\widetilde Q)\). 

% Taking into account of the boundary condition appearing in the transmission eigenvalue problem, we first define the finite-dimensional approximation space $V_h$
% \begin{equation}
%     V_h = \mathcal{S}_h \cap H_0^2(\Omega),
% \end{equation}
% where the spline space $\mathcal{S}_h$ is defined in \eqref{eq:S_h},  

The finite-dimensional approximation space of the Galerkin method for the variational form \eqref{eq:linear_eigenvalue_problem} is 
\begin{equation}
    \mathbb{H}_h:=V_h\times V_h \subset \mathbb{H} =H_0^2(\Omega)\times L^2(\Omega) , 
\end{equation}
and the discrete variational form of \eqref{eq:linear_eigenvalue_problem} reads: Find $\big( \lambda_h, (u_h,z_h) \big)\in \mathbb{C}\times \mathbb{H}_h$ with $\lambda_h \ne 0$ and $u_h\ne 0$,  such that 
\begin{equation}\label{eq:Discrete_eigenvalue_problem}
    A\big( (u_h,z_h),(v_h,w_h)   \big) = \lambda_h B\big( (u_h,z_h),(v_h,w_h)  \big)  ~~\quad \forall (v_h,w_h)\in \mathbb{H}_h. 
\end{equation}

We define the discrete solution operator $T_h: \mathbb{H}\to \mathbb{H}_h$ as follows: For any $(f,g)\in \mathbb{H}$, we have
\begin{equation}\label{eq:T_h_def}
    T_h(f,g) = (\widetilde{f}_h,\widetilde{g}_h),
\end{equation}
where $(\widetilde{f}_h,\widetilde{g}_h)$ is the unique solution to the discrete source problem: For any given $(f,g)\in \mathbb{H}$, find $(\widetilde{f}_h,\widetilde{g}_h)\in \mathbb{H}_h$ such that
\begin{equation}\label{eq:Discrete_Source_Problem}
A\big((\widetilde{f}_h,\widetilde{g}_h), (v_h,w_h)\big) = B\big( (f,g), (v_h,w_h)\big)\qquad \forall (v_h,w_h) \in \mathbb{H}_h.
\end{equation}

It follows from \eqref{eq:Discrete_eigenvalue_problem}, \eqref{eq:T_h_def}, and \eqref{eq:Discrete_Source_Problem} that $\big(\lambda_h^{-1}, (u_h,z_h) \big)$ is an eigenpair of $T_h$, that is
\begin{equation}
    T_h (u_h,z_h) = \lambda_h^{-1} (u_h,z_h).
\end{equation}

Let $T_h^*: \mathbb{H}\to \mathbb{H}_h$ be the discrete adjoint operator of $T_h$, which is defined by $T_h^*(f,g):=(\widetilde{f}_h^*, \widetilde{g}_h^*)$, where $(f,g) \in \mathbb{H}$, and  $(\widetilde{f}_h^*, \widetilde{g}_h^*)\in \mathbb{H}_h$ is the unique solution to the following discrete source problem
\begin{equation}
    A\big( (v_h,w_h), (\widetilde{f}_h^*, \widetilde{g}_h^*)  \big) =  B\big( (v_h,w_h), (f,g) \big) ~~~~ \quad  \forall (v_h,w_h)\in \mathbb{H}_h.
\end{equation}

\subsection{Approximation with B-splines in the physical domain}

Since in our GIFT scheme, we utilize the B-spline basis functions, which are obtained by performing a certain refinement to the initial knot vectors $\Xi$ and $\mathcal{H}$, to construct the finite-dimensional approximation space $V_h$, we need to establish the approximation theory of B-splines in the physical domain.

Let the bivariate B-spline space over the parametric domain  $\widehat{\Omega} = (0,1)^2$ be
 \begin{equation}
    \mathcal{S}_h := \emph{span}\{ B_{i,j}^{(p,q)}(\xi, \eta)  \}_{1\le i \le n_\xi,\, 1\le j \le n_\eta },  
 \end{equation} 
following \cite{Bazilevs2006}, we define the projector $\Pi_{S_h}: L^2(\widehat{\Omega})\to \mathcal{S}_h$ as
\begin{equation}
\Pi_{S_h} v := \sum_{i=1}^{n_\xi} \sum_{j=1}^{n_\eta} \lambda_{ij}(v)B_{i,j}^{(p,q)} \qquad \forall v\in L^2(\widehat{\Omega}),    
\end{equation}
where $\lambda_{ij}$ are the dual basis functionals with respect to $\{ B_{i,j}^{(p,q)} \}_{1\le i \le n_\xi, \, 1\le j \le n_\eta}$, i.e.,
\begin{equation}
\lambda_{ij}(B_{i'j'}) = 
\begin{cases}
    1 \quad & \mbox{ if  }  \quad i = i' \mbox{  and  } j = j',\\
    0 \quad & \mbox{ otherwise.  }
\end{cases}
\end{equation}

Let the  spline space over the physical domain $\Omega$ be \begin{equation}\label{eq:V_h}
    \widetilde{V}_h := \emph{span}\{ B_{i,j}^{(p,q)}\circ {\bf{F}}^{-1}, \,\, 1\le i \le n_\xi,\,\, 1\le j \le n_\eta \},
\end{equation}
we define the projector $\Pi_{\widetilde{V}_h}: L^2( \Omega)\to \widetilde{V}_h$ by
\begin{equation}
    \Pi_{\widetilde{V}_h} v := \big( \Pi_{S_h}(v\circ \mathbf{F}) \big) \circ  \mathbf{F}^{-1} \qquad \forall v \in L^2(\Omega).
\end{equation}

Concerning the error between $v$ and $\Pi_{\widetilde{V}_h}v$, we have the following result.
\begin{theorem}\label{theo:local_projection_error}
(Local projection error estimate). Let $\ell$ and $s$ be two integers such that $0\le \ell \le s \le p+1$, and let $\widehat{K}\in \widehat{\mathcal{T}}_h$, and $K = \mathbf{F}(\widehat{K})$, 
% there exists a constant $C_{shape}$ such that
we have
\begin{equation}\label{inequality:Local_Projection_error}
    | v - \Pi_{\widetilde{V}_h} v  |_{\ell,K} \le C_{shape} h_K^{s-\ell}\sum_{i=0}^s \| \nabla \mathbf{F} \|_{L^{\infty}(\widetilde{\widehat{K}})}^{i-s} | v  |_{i,\widetilde{K}} \qquad  \forall v \in L^2(\Omega)\cap H^{s}(\widetilde{K}),
\end{equation}
where $C_{shape}$ is a positive constant depending on the shape of $\Omega$, but independent of the mesh size $h$, $\widetilde{\widehat{K}}$ is the support  extension \cite{Bazilevs2006} of   $\widehat{K}$, and $\widetilde{K} = \mathbf{F}(\widetilde{\widehat{K}})$. 
\begin{proof}
By employing Lemmas 3.3 and 3.5 of \cite{Bazilevs2006}, the estimate \eqref{inequality:Local_Projection_error} can be proved analogously to Theorem 3.1 in \cite{Bazilevs2006} with a minor modification.
\end{proof}
\end{theorem}

Let $m_{\xi}$ and $m_{\eta}$ be the numbers of distinct knots in the knot vectors $\Xi$ and $\mathcal{H}$, respectively, 
and let the ordered distinct knots in $\Xi$ and $\mathcal{H}$  form the sets 
$\{\widetilde{\xi}_1, \widetilde{\xi}_2, \ldots, \widetilde{\xi}_{m_\xi}\}$ and 
$\{\widetilde{\eta}_1, \widetilde{\eta}_2, \ldots, \widetilde{\eta}_{m_\eta}\}$. 
Their corresponding multiplicities are denoted by 
$\{k_i^{(\xi)}\}_{i=1}^{m_\xi}$ and $\{k_j^{(\eta)}\}_{j=1}^{m_\eta}$. Since
the  multiplicities of internal knots determine the continuity of B-spline basis functions, we define the integer $r_{\min}$ as
\[
r_{\min} := \min\{r_{\min}^{(\xi)},\, r_{\min}^{(\eta)}\},
\]
where
\[
r_{\min}^{(\xi)} := \min_{2 \le i \le m_{\xi} - 1} (p - k_i^{(\xi)}), \qquad
r_{\min}^{(\eta)} := \min_{2 \le j \le m_{\eta} - 1} (q - k_j^{(\eta)}).
\]
Then we have $\widetilde{V}_h \subset C^{\,r_{\min}}(\Omega)$, and therefore the inclusion $\widetilde{V}_h \subset H^{r_{\min} + 1}(\Omega)$ holds.

\begin{theorem}\label{theo:global_projection_error}
(Global projection error estimate). Let $\ell$ and $s$ be two integers such that $0\le \ell  \le p+1$, and $\ell \le s$. If  $\ell \le r_{min} + 1$, then
% there exists a constant $C_{shape}$ such that
we have
\begin{equation}\label{inequality:Global_Projection_error}
    | v - \Pi_{\widetilde{V}_h} v  |_{\ell,\Omega} \le C_{shape} h^{\min\{p+1,s\} - \ell}\|v\|_{s,\Omega} \qquad  \forall v \in  H^{s}(\Omega).
\end{equation}
\begin{proof}
Since $\ell \le r_{min}+1$, we have $\Pi_{\widetilde{V}_h}v \in H^{\ell}(\Omega)$. Therefore, the semi-norm $|v - \Pi_{\widetilde{V}_h}v|_{\ell, \Omega}$ is well-defined. By using Theorem \ref{theo:local_projection_error} and the standard arguments as in Corollary 3.1 of \cite{TAGLIABUE2014277}, we can prove the error estimate \eqref{inequality:Global_Projection_error} immediately.
\end{proof}
\end{theorem}

Lemma~\ref{lemma:regularity_result} implies that the solution to the source problem~\eqref{eq:source_problem} may be of low regularity: the solution lies in $H^{2+\sigma}(\Omega)\times H_0^2(\Omega)$, where $\sigma \in (0,1]$, we therefore need to estimate the projection error for functions  $v$ belonging  to the fractional-order Sobolev space, that is, $v\in H^{2+\sigma}(\Omega)$ with $\sigma\in(0,1]$. For this case, we set the degrees of B-splines to  $p=q=2$. 

\begin{theorem}\label{theo:projection_error_fractional}
 Let $\ell$ be an integer such that $0\le \ell  \le 2$. If  $\ell \le r_{min} + 1$, then
% there exists a constant $C_{shape}$ such that
we have
\begin{equation}\label{inequality:Projection_error_fractional}
    \| v - \Pi_{\widetilde{V}_h} v  \|_{\ell,\Omega} \le C_{shape} h^{ 2+ \sigma  - \ell}\|v\|_{2+\sigma,\Omega} \qquad  \forall v \in  H^{2+\sigma}(\Omega) \quad \mbox{with} \quad \sigma \in (0,1].
\end{equation}

\begin{proof}
We prove this theorem by  considering the following two cases.

\begin{enumerate}
    \item[(1)]  If $\sigma = 1$, then the error estimate \eqref{inequality:Projection_error_fractional} follows from Theorem \ref{theo:global_projection_error} 
with $p=2$ and $s=3$.
  \item[(2)] If $\sigma\in (0,1)$, we employ a standard Banach-space interpolation argument. 
Let $E_0 := H^2(\Omega)$ and $E_1 := H^3(\Omega)$. By the reiteration theorem for real interpolation, we have
\[
[E_0, E_1]_{\sigma, 2} = H^{2+\sigma}(\Omega).
\]
Consider the linear operator
\[
T = I - \Pi_{\widetilde{V}_h} : H^s(\Omega) \to H^\ell(\Omega), \quad  \mbox{for} \quad s = 2, 3.
\]
From Theorem \ref{theo:global_projection_error}, we obtain that
\[
\|T\|_{\mathcal{L}(H^2(\Omega), H^\ell(\Omega))} \le C_{\rm shape} h^{2-\ell}, \qquad  \mbox{and} \qquad
\|T\|_{\mathcal{L}(H^3(\Omega), H^\ell(\Omega))} \le C_{\rm shape} h^{3-\ell}.
\]
By using the Banach-space interpolation theory (see, e.g., \cite{brenner08}), we can conclude that
\[
\|T\|_{\mathcal{L}(H^{2+\sigma}(\Omega), H^\ell(\Omega))}
\le C_{\rm shape} h^{(2-\ell)(1-\sigma) + (3-\ell)\sigma}
= C_{\rm shape} h^{2+\sigma - \ell}.
\]
Consequently, we have
\[
\| v - \Pi_{\widetilde{V}_h} v \|_{\ell,\Omega}
\le C_{\rm shape} h^{2+\sigma - \ell} \|v\|_{2+\sigma,\Omega}.
\]
% which proves the desired estimate for $0<\sigma<1$. 
\end{enumerate}
% The proof is complete.

\end{proof}
\end{theorem}

We next consider the approximation property of the finite-dimensional space $V_h = \widetilde{V}_h \cap H_0^2(\Omega)$, in which the essential boundary conditions are imposed.

Let the projector $\Pi_{\mathcal{S}_h}^0: H_0^2(\widehat{\Omega}) \to \mathcal{S}_h \cap H_0^2(\widehat{\Omega})$ be defined by
\begin{equation}
    \Pi_{\mathcal{S}_h}^0 = 
    \sum_{\substack{1\le i \le n_{\xi}, 1\le j \le n_{\eta}\\
    B_{i,j}^{(p,q)}\in H_0^2(\widehat{\Omega})}}
    \lambda_{ij}(v)B_{i,j}^{(p,q)} \qquad \forall v\in H_0^2(\widehat{\Omega}), 
\end{equation}
and we define  the projector $\Pi_{V_h}: H_0^2(\Omega) \to V_h$ as 
\begin{equation}
    \Pi_{{V}_h} v := \big( \Pi_{S_h}^0(v\circ \mathbf{F}) \big) \circ  \mathbf{F}^{-1} \qquad \forall v \in H_0^2(\Omega).
\end{equation}

The approximation property of $V_h$ is established in the following theorem, whose proof is similar to that of Theorem \ref{theo:global_projection_error}.

\begin{theorem}\label{theo:global_projection_error_boundary}
Let $\ell$ and $s$ be two integers such that $0\le \ell  \le p+1$, and $\ell \le s$. If  $\ell \le r_{min} + 1$, then
% there exists a constant $C_{shape}$ such that
we have
\begin{equation}\label{inequality:Global_Projection_error_boundary}
    | v - \Pi_{{V}_h} v  |_{\ell,\Omega} \le C_{shape} h^{\min\{p+1,s\} - \ell}\|v\|_{s,\Omega} \qquad  \forall v \in  H^{s}(\Omega)\cap H_0^2(\Omega).
\end{equation}
\end{theorem}

Similar to Theorem \ref{theo:projection_error_fractional}, we can prove the following result.

\begin{theorem}\label{theo:projection_error_fractional_boundary}
 Let $\ell$ be an integer such that $0\le \ell  \le 2$. If  $\ell \le r_{min} + 1$, then
% there exists a constant $C_{shape}$ such that
we have
\begin{equation}\label{inequality:Projection_error_fractional_boundary}
    \| v - \Pi_{{V}_h} v  \|_{\ell,\Omega} \le C_{shape} h^{ 2+ \sigma  - \ell}\|v\|_{2+\sigma,\Omega} \qquad  \forall v \in  H^{2+\sigma}(\Omega) \cap H_0^2(\Omega) \quad \mbox{with} \quad \sigma \in (0,1].
\end{equation}
\end{theorem}

\section{Spectral approximation and error estimates}\label{sec:error_estimate}

In this section, we derive the optimal error estimate of isogeometric discretization for the 
linear transmission eigenvalue problem \eqref{eq:linear_eigenvalue_problem}.
We first prove the convergence of $T_h$ to $T$ in operator norm as the global mesh size $h$ tends to zero.

\begin{lemma} 
\label{lemma:convergence of operator}
There exists a positive constant $C$ depending on the index of refraction $n$ and the shape of domain $\Omega$, but independent of $h$,  such that
\begin{equation}\label{ineq:T_T_h_error}
       \| T_h - T \|_{\mathcal{L}(\mathbb{H})} \le C h^{\sigma}, \quad \mbox{for} \quad h \le 1, 
\end{equation}
where $\sigma\in (0,1]$ is the parameter appearing in Lemma~\ref{lemma:regularity_result}.
\begin{proof}
    % Let $(f,g)\in \mathbb{H}$ such that $T(f,g)=(\Tilde{u},\Tilde{z})$ and $T_h(f,g)=(\Tilde{u}_h,\Tilde{z}_h)$ with
    % \begin{align}
    %     A(T(f,g),(v,w)) &= B((f,g),(v,w)),\quad \forall (v,w)\in \mathbb{H}, \label{Operator sesqualinear form} \\
    %     A(T_h(f,g),(v_h,w_h)) &= B((f,g),(v_h,w_h)),\quad \forall (v_h,w_h)\in \mathbb{H}_h.
    % \end{align}
    % Since $\mathbb H_h\subset \mathbb H$, set $(v_h,w_h)$ as the second component in Eq.\ref{Operator sesqualinear form} and subtracting gives
    For any given $(f,g)\in \mathbb{H}$, using the Galerkin orthogonality property
    \begin{align}
        A\big(\,T(f,g)-T_h(f,g),(v_h,w_h)\,\big) = 0 \qquad \forall (v_h,w_h)\in\mathbb H_h\subset \mathbb{H},
    \end{align}
    and applying the coercivity and boundedness of the sesquilinear form $A(\cdot,\cdot)$ (see Lemma \ref{lemma_A_B}), we have
    \begin{equation}
    \begin{aligned}
        \beta \| T(f,g) - T_h(f,g)  \|_{\mathbb{H}}^2 \le & A\big(\, T(f,g) - T_h(f,g), T(f,g) - T_h(f,g)\, \big) \\
        = & A\big(\, T(f,g) - T_h(f,g), T(f,g) - (v_h,w_h)\, \big) \\
        \le & C \|  T(f,g) - T_h(f,g)  \|_{\mathbb{H}} \, \|  T(f,g) - (v_h,w_h) \|_{\mathbb{H}} \quad \forall (v_h,w_h)\in \mathbb{H}_h,
    \end{aligned}
    \end{equation}
where $\beta$ and $C$ are the constants appearing in the coercivity and boundedness of the sesquilinear form $A(\cdot,\cdot)$, respectively.
Consequently,  we obtain the following error estimate result
\begin{equation}\label{inequality_Cea_Lemma}
\|  T(f,g) - T_h(f,g)  \|_{\mathbb{H}} \le \frac{C}{\beta}\inf\limits_{(v_h,w_h)\in \mathbb{H}_h} \|  T(f,g) - (v_h,w_h) \|_{\mathbb{H}},  
\end{equation}

It follows from  Lemma  \ref{lemma:regularity_result} that $(\widetilde{f},\widetilde{g})= T(f,g)\in H^{2+\sigma}(\Omega)\times H_0^2(\Omega)$, where $\sigma\in (0,1]$. For the inequality \eqref{inequality_Cea_Lemma}, if we set $(v_h,w_h)= ( \Pi_{ {V}_h} \widetilde{f}, \Pi_{{V}_h} \widetilde{g})\in \mathbb{H}_h$, and use the spline approximation result \eqref{inequality:Projection_error_fractional_boundary},  then we can obtain that
\begin{equation}\label{ineq:T_T_h_f_g}
\begin{aligned}
 \|  T(f,g) - T_h(f,g)  \|_{\mathbb{H}} & \le \frac{C}{\beta}\| (\widetilde{f},\widetilde{g}) -  ( \Pi_{V_h} \widetilde{f}, \Pi_{V_h} \widetilde{g}) \|_{\mathbb{H}} \le  \frac{C}{\beta} \big(  \|  \widetilde{f} - \Pi_{V_h} \widetilde{f} \|_{2,\Omega} +   \|  \widetilde{g} - \Pi_{V_h} \widetilde{g} \|_{0,\Omega} \big)\\
 &\le \frac{C}{\beta} \big(   C_{shape} h^{\sigma} \| \widetilde{f} \|_{2+\sigma,\Omega} + C_{shape} h^2 \|  \widetilde{g} \|_{2,\Omega} \big). 
 \end{aligned}
\end{equation}

By using  Lemma \ref{lemma:regularity_result} and  inequality \eqref{ineq:T_T_h_f_g}, we can obtain the error estimate \eqref{ineq:T_T_h_error}.

\end{proof}
\end{lemma}

Similarly, we can also prove the convergence of $T_h^*$ to $T^*$ in the operator norm as the mesh size $h$ goes to zero.

\begin{lemma}
\label{lemma:convergence of adjoint operator}
There exists a positive constant $C$ depending on the index of refraction $n$ and the shape of domain $\Omega$, but independent of $h$, such that
\begin{equation}
       \|  T_h^* - T^*  \|_{\mathcal{L}(\mathbb{H})} \le C h^{\sigma},  
\end{equation}
where $\sigma\in (0,1]$ is the parameter appearing in Lemma~\ref{lemma:regularity_result}.
\end{lemma}

% We are now in a position to prove the optimnal error estimates for the eigenvalues and eigenfunctions. % To do this, we will apply the theory of Babuška and Osborn \cite{BABUSKA1991641}.

Let $\mu:=\lambda^{-1}$ be a non-zero eigenvalue of the operator $T$ with algebraic multiplicity $m$, and let $\Gamma$ be a circle in the complex plane centered at $\mu$ that encloses no other eigenvalues. We recall that the Riesz spectral projector associated with $T$ and $\mu$ is defined by 
\begin{equation}
    E = E(\mu) =  \frac{1}{2\pi i}\int_{\Gamma} (z - T)^{-1}{\rm d}z.
\end{equation}
This operator is a projection onto the space of generalized eigenvectors related to $T$ and $\mu$, that is, $R(E) = N((\mu - T)^\alpha)$, where $\alpha$ is the ascent of $\mu - T$, and $R(\cdot)$ and $N(\cdot)$ denote the range and kernel, respectively. Therefore, we have dim$R(E) = m$, and it follows from Lemma \ref{lemma:regularity_result} that $R(E)\subset H^{2+\sigma}(\Omega)\times H_0^2(\Omega) $.

Similarly, let $\Gamma^*$ be a circle in the complex plane centered at $\overline{\mu}$ that encloses no other eigenvalues, the spectral projector $E^*$ related to the dual solution operator $T^*$ and $\overline{\mu}$ is defined as
\begin{equation}
   E^* =  \frac{1}{2\pi i}\int_{\Gamma^{*}} (z - T^*)^{-1}{\rm d}z,
\end{equation}
and it holds that $R(E^*)\subset H^{2+\sigma}(\Omega)\times H_0^2(\Omega) $.
% denoted by $R(E)$ and $R(E^*)$, where $R$ denotes the range.

Since $T_h \to T$ in operator norm as $h\to 0$, there are $m$ eigenvalues $\mu_{1,h},\ldots, \mu_{m,h}$ of $T_h$ that converge to $\mu$ as $h\to 0$.
For $h$ sufficiently small, we define the Riesz spectral projection associated with $T_h$ and the eigenvalues of $T_h$ lying in $\Gamma$ by
\begin{align}
    E_h  =  \frac{1}{2\pi i}\int_{\Gamma} (z-T_h)^{-1}\,{\rm d}z,
\end{align}
which  is a projection onto the space spanned by the spaces of generalized eigenvectors of $T_h$ corresponding to $\mu_{1,h},\ldots, \mu_{m,h}$. 
% Although each eigenvalue $\mu_{h,j}$, $j=1,\ldots,m$, converges to $\mu$
%as $h\to0$, their arithmetic mean generally yields a closer approximation
% to \(\mu\), see~\cite{BABUSKA1991641}. 

% \textcolor{red}{The definition of $R(E_h)$???}

We recall the definition of the gap between closed subspaces. Let
$\mathbb X$ and $\mathbb Y$ be closed subspaces of the Hilbert space $\mathbb H$, the gap
between $\mathbb X$ and $\mathbb Y$ is defined by
\begin{align}
    \widehat{\delta}(\mathbb X,\mathbb Y)
    :=
    \max\{\delta(\mathbb X,\mathbb Y),\delta(\mathbb Y,\mathbb X)\},
\end{align}
where
\[
    \delta(\mathbb X,\mathbb Y)
    :=
    \sup_{\substack{x\in \mathbb X\\ \|x\|_{\mathbb H}=1}}
    \delta(x,\mathbb Y), \quad \mbox{and} 
    \quad
    \delta(x,\mathbb Y)
    :=
    \inf_{y\in \mathbb Y}\|x-y\|_{\mathbb H}.
\]

The optimal error estimates for the eigenvalues and  eigenfunctions are established in the following theorem.

\begin{theorem}
    There exists a positive constant $C$ depending on the index of refraction $n$ and the shape of domain $\Omega$, but independent of $h$,  such that
\begin{align}
\widehat\delta(R(E),R(E_h))&\leq C h^{\sigma}, \label{eq:eigenfunction convergence order} \\  
|\mu - \widehat\mu_h| &\le C h^{2\sigma}, \label{eq:eigenvalue convergence order}
\end{align}
where $\sigma\in (0,1]$ is the parameter appearing in Lemma~\ref{lemma:regularity_result}, and 
\begin{align}
    \widehat{\mu}_h
    =
    \frac{1}{m}\sum_{j=1}^{m}\mu_{h,j}.
\end{align}
\begin{proof}
It follows from Theorem 7.1 of \cite{BABUSKA1991641} that there exists a constant $C$ independent of $h$, such that
\begin{align}\label{ineq:delta_R_E}
\widehat\delta(R(E),R(E_h)) \leq C \| (T - T_h)|_{R(E)} \|.
    \end{align}
    % Eq.(\ref{eq:eigenfunction convergence order}) is obtained from Eq.(\ref{eq:inequality_proof_estimate}) by taking $(f,g)\in R(E)$.
Combining \eqref{ineq:T_T_h_error}  and \eqref{ineq:delta_R_E}, we can obtain the error estimate \eqref{eq:eigenfunction convergence order}. 

Next, we proceed to prove \eqref{eq:eigenvalue convergence order}. Since the sesquilinear form $A(\cdot,\cdot)$ defines an inner product on the Hilbert space $\mathbb{H}$, if $\{(u_i,z_i)\}_{i=1}^m$ forms a basis for $R(E)$, then there exists a unique dual basis $\{(u_i^*,z_i^*)\}_{i=1}^m \subset R(E^*)$,  satisfying
\begin{equation}
  \langle (u_i,z_i), (u_j^*,z_j^*) \rangle = A\left((u_i,z_i), (u_j^*,z_j^*)\right) = \delta_{ij},
  \qquad i,j = 1,\ldots,m,
\end{equation}
where $\langle \cdot, \cdot \rangle $ is the Hilbert space duality pairing.

From Theorem 7.2 of \cite{BABUSKA1991641}, we have 
\begin{align}\label{inequality:mu_mu_h_error}
        |\mu-\widehat\mu_h|
  \leq  \frac{1}{m}
   \Big(
  \sum_{j=1}^{m}|\langle ( T-T_h)(u_j,z_j),(u_j^*,z_j^*)\rangle|
  +
  \| (T-T_h)|_{R(E)} \|\,
  \| (T^*-T_h^*)|_{R(E^*)} \|
  \Big).
    \end{align}
    
On the one hand, it follows from Lemmas \ref{lemma:convergence of operator} and \ref{lemma:convergence of adjoint operator} that 
\begin{equation}\label{T_Th_T*_Th*_error}
     \| (T-T_h)|_{R(E)} \|\,
  \| (T^*-T_h^*)|_{R(E^*)} \| \le C h^{2\sigma}.
\end{equation}

On the other hand, by the Galerkin orthogonality property and the boundedness of $A(\cdot,\cdot)$, we have 
\begin{equation}\label{inequality:T_T_h_dual}
\begin{aligned}
    |\langle (T-T_h)(u_j,z_j),(u_j^*,z_j^*) \rangle | &= |A((T-T_h)(u_j,z_j),(u_j^*,z_j^*))| \\
    & = |A((T-T_h)(u_j,z_j),(u_j^*,z_j^*)-(\Pi_{h}u_j^*,\Pi_{h}z_j^*))| \\
    &\leq C \| (T-T_h)(u_j,z_j) \|_{\mathbb H}\| (u_j^*,z_j^*)-(\Pi_{h}u_j^*,\Pi_{h}z_j^*) \|_{\mathbb H} \\
    &= C\| (T-T_h)(u_j,z_j)\|_{\mathbb H}\|(u_j^*-\Pi_{h}u_j^*,z_j^*-\Pi_{h}z_j^*)\|_{\mathbb H} \\
    &\le C \|  T - T_h \|_{\mathcal{L(\mathbb{H})}}\| (u_j,z_j) \|_{\mathbb{H}}\cdot \big( \| u_j^*-\Pi_{V_h}u_j^*  \|_{2,\Omega} + \|  z_j^*-\Pi_{V_h}z_j^* \|_{0,\Omega} \big).
\end{aligned}
\end{equation}

The error estimate \eqref{eq:eigenvalue convergence order} follows from \eqref{inequality:mu_mu_h_error}, \eqref{T_Th_T*_Th*_error},  \eqref{inequality:T_T_h_dual}, and  Theorem \ref{theo:projection_error_fractional_boundary} and Lemma \ref{lemma:convergence of operator}.
\end{proof}
\end{theorem}
% \begin{remark}
%     As stated in Remark.\ref{Remark for B-spline degree}, it is valid for the convergence of adjoint operator as well. As a consequence, we have
%     \begin{align}
%         \widehat\delta(R(E),R(E_h))&\leq C_n h^{\min\{p-1,s\}},\label{eq:eigenfunction convergence order} \\
%            |\mu - \widehat\mu_h| &\le C_n h^{\min\{2(p-1),2s\}}.
%     \end{align}

Note that the above optimal error estimates are derived under the assumption that the
solution of the source problem \eqref{eq:source_problem} has low
regularity, that is, $(\widetilde{f},\widetilde{g}) \in H^{2+\sigma}(\Omega)\times H_0^2(\Omega)$ with $0 < \sigma \le 1$. The error estimates for eigenvalues and eigenfunctions of the transmission eigenvalue problem with low-regularity solutions have also been analyzed in the framework of $C^1$ virtual element methods
\cite{Mora2018,mora2021virtual}. However, the high-regularity case was not considered there.
For the source problem \eqref{eq:source_problem}, when the functions $f$ and $g$, the index of refraction $n$, and the boundary of domain $\Omega$ are sufficiently smooth, its solution $(\widetilde{f},\widetilde{g})$ would also be sufficiently smooth. In this respect, we have the following theorem.

\begin{theorem}
Assume that there exists an integer $r\ge p+1 $  such that
\begin{equation}\label{inclusion:assumption}
R(E),\,\, R(E^*) \subset H^{r}(\Omega) \times H^{r}(\Omega), 
    \end{equation}
then we have
\begin{equation}
\widehat\delta(R(E),R(E_h)) \leq C h^{ p-1 },
\end{equation}
and
\begin{equation}
\label{eq:eigenfunction convergence order smooth} 
|\mu - \widehat\mu_h|  \le C h^{2  (p-1) },
\end{equation}
where  $C$ is a constant depending on the index of refraction and the domain $\Omega$, but independent of $h$.
\begin{proof}
%  By the fact $R(E)\subset R(T)$ and \eqref{inclusion:assumption}, we have
% \begin{equation}
% R(E)\subset H^{r}(\Omega)\times H^{r}(\Omega),\qquad r\ge 2+\sigma.
% \end{equation}
For any $(u,z)\in R(E)$, we set $(\widetilde u,\widetilde z):=T(u,z)$. Since $R(E)$ is invariant under $T$, we have $(\widetilde u,\widetilde z)\in R(E)$.

On the finite-dimensional space $R(E)$ (the dimension of $R(E)$ is $m$), all norms are equivalent. Hence there exists a constant $C>0$, independent of $h$, such that
\begin{equation}\label{ineq:norm_equivalence}
\|(\widetilde u,\widetilde z)\|_{H^r(\Omega)\times H^r(\Omega)}\le C\|(\widetilde u,\widetilde u)\|_{\HH}= C\|T(u,z)\|_{\HH}.
\end{equation}
By \eqref{ineq:norm_equivalence} and the boundedness of $T:\HH\to\HH$, we obtain that
\begin{equation}\label{inequa:f_g_u_z_bound}
\|(\widetilde u,\widetilde z)\|_{H^r(\Omega)\times H^r(\Omega)}
\le C\|T\|_{\LL(\HH)}\|(u,z)\|_{\HH}. 
\end{equation}

Now, for any $(u,z)\in R(E)$, let $(\widetilde u_h,\widetilde z_h):=T_h(u,z)$.  Similar to the derivation of \eqref{ineq:T_T_h_f_g}, we can derive that
\begin{equation}\label{ineq:T_T_h_f_g_smooth}
\begin{aligned}
 \|  T(u,z) - T_h(u,z)  \|_{\mathbb{H}}  \le \frac{C}{\beta} \big(   C_{shape} h^{ \min\{ p+1,r \}-2 } \| \widetilde{u} \|_{r,\Omega} + C_{shape} h^{\min\{ p+1,r \}} \|  \widetilde{z} \|_{r,\Omega} \big). 
 \end{aligned}
\end{equation}

%The Galerkin orthogonality and the coercivity of $A(\cdot,\cdot)$ imply
% \begin{equation}
% \|T(u,z)-T_h(u,z)\|_{\HH}
% \le C\inf_{(v_h,w_h)\in\HH_h}\|T(u,z)-(v_h,w_h)\|_{\HH}. \tag{2}
% \end{equation}
% Choosing $(v_h,w_h)=(\Pi_h\widetilde f,\Pi_h\widetilde g)$, where $\Pi_h$ is the B-spline projector from Theorem~3.2, we obtain
% \begin{equation}
% \begin{aligned}
% \|T(u,z)-T_h(u,z)\|_{\HH}
% &\le C\left(\|\widetilde f-\Pi_h\widetilde f\|_{2,\Omega}
% +\|\widetilde g-\Pi_h\widetilde g\|_{0,\Omega}\right).
% \end{aligned} \tag{3}
% \end{equation}
% By Theorem~3.2, for $\ell=2$ and $\ell=0$, respectively,
% \begin{equation}
% \|\widetilde f-\Pi_h\widetilde f\|_{2,\Omega}
% \le C h^{\min\{p+1,r\}-2}\|\widetilde f\|_{r,\Omega},
% \end{equation}
% \begin{equation}
% \|\widetilde g-\Pi_h\widetilde g\|_{0,\Omega}
% \le C h^{\min\{p+1,r\}}\|\widetilde g\|_{r,\Omega}.
% \end{equation}
% Since $r\ge 2+\sigma>2$ and $p\ge2$, we have
% \begin{equation}
% \min\{p+1,r\}\ge \min\{p-1,r-2\}+2,
% \end{equation}
% so both estimates are bounded by
% \begin{equation}
% C h^{\min\{p-1,r-2\}}\left(\|\widetilde f\|_{r,\Omega}+\|\widetilde g\|_{r,\Omega}\right).
% \end{equation}
Combining \eqref{inequa:f_g_u_z_bound} and \eqref{ineq:T_T_h_f_g_smooth}, we have
\begin{equation}
\|T(u,z)-T_h(u,z)\|_{\HH}
\le C h^{\min\{p-1,r-2\}}\|(u,z)\|_{\HH}, 
\end{equation}
then we can conclude that
\begin{equation}\label{inequality:T_smooth}
\|(T-T_h)|_{R(E)}\|_{\LL(\HH)}
\le C h^{\min\{p-1,r-2\}}.  
\end{equation}
The same argument applies to the adjoint solution operator $T^*$, yielding
\begin{equation}\label{inequality:T*_smooth}
\|(T^*-T_h^*)|_{R(E^*)}\|_{\LL(\HH)}
\le C h^{\min\{p-1,r-2\}}.  
\end{equation}

By using \eqref{ineq:delta_R_E},  \eqref{inequality:T_smooth}, and the assumpation $r\ge p+1$, we can derive that
\begin{equation}
\widehat{\delta}(R(E),R(E_h)) \le C h^{p-1}.
\end{equation}

By combining \eqref{inequality:mu_mu_h_error}, \eqref{inequality:T_smooth}, \eqref{inequality:T*_smooth}, \eqref{inequality:T_T_h_dual}, and Theorem \ref{theo:global_projection_error_boundary}, we can achieve the error estimate \eqref{eq:eigenfunction convergence order smooth}.

\end{proof}
\end{theorem}

%  \newpage
\section{Numerical examples}\label{sec:examples}

In this section, we present several two- and three-dimensional numerical examples to validate our theoretical results, and show the advantages of IGA over some existing numerical methods for the transmission eigenvalue problem. Since the exact transmission eigenvalues are unavailable, following \cite{yang2016mixed,meng2023mixed}, we use the convergence of the relative error $R.E.(h)$ for the transmission eigenvalues with respect to the mesh size $h$  to estimate the convergence order, where 
the relative error $R.E.(h)$ is defined as
\[
R.E.(h_i) = \frac{|k_{j,h_i}-k_{j,h_{i+1}}|}{|k_{j,h_{i+1}}|},\]
with $k_{j,h_i}$ being the $j$-th transmission eigenvalue achieved by IGA with physical mesh size $h_i$, and $h_i>h_{i+1}$. The computed convergence order is given by
\[
\mbox{Order} = \frac{\log\big(\, R.E.(h_i) / R.E.(h_{i+1})\, \big)}{\log (h_i/h_{i+1})}.
\]
For all the following numerical examples, the mesh of the physical domain is obtained by a uniform knot refinement in the parametric domain, which is then geometrically mapped to the physical domain via the NURBS geometric mapping.

\subsection{Example 1: Square domain with constant index of refraction}

In the first numerical example, we consider the transmission eigenvalue problem in a square domain $\Omega = (0,1)^2$, where the index of refraction is $n=4$. We show the first four transmission eigenvalues computed using the isogeometric method on five successively refined meshes with degrees $p=2$ and $p=3$ in Tables \ref{tab:k-case1-p2} and \ref{tab:k-case1-p3}, respectively. We observe that the computed eigenvalues agree well with those reported in \cite{meng2023mixed,HAN201796,mora2021virtual,Mora2018}. Furthermore, it is found that the computed real eigenvalues are monotonically decreasing as the mesh size $h$ becomes smaller, which implies that the convergence of our method. The convergence of the relative error of the first four transmission eigenvalues with respect to  $h$ is displayed in Figure~\ref{fig:example_1_n_4}, which successfully verifies our theoretical result, i.e., the convergence order of the error in eigenvalues is $2(p-1)$, where $p\ge 2$.

\begin{table}[!htbp]
\centering
\caption{The first four eigenvalues obtained by IGA with $p=2$, where $\Omega = (0,1)^2$ and $n = 4$.}
% \tiny
\label{tab:k-case1-p2}
% \resizebox{\textwidth}{!}{%
\begin{tabular}{ccccc}
\toprule
$h$ & DOFs & $k_{1,h}\,\, \&\,\, k_{2,h}$  & $k_{3,h}$ & $k_{4,h}$ \\
\midrule
$\frac{\sqrt{2}}{16}$ &  648 &  4.265612797601  $\pm$  1.157188956535i  &  5.631839451281  &   5.631839451281  \\ 
$\frac{\sqrt{2}}{32}$ &  2312 &  4.269984960633  $\pm$  1.149821012738i  &  5.512960721511  &   5.512960721511  \\ 
$\frac{\sqrt{2}}{64}$ &  8712 &  4.271256693144  $\pm$  1.148026534850i  &  5.485207606174  &   5.485207606185  \\ 
$\frac{\sqrt{2}}{128}$ &  33800 &  4.271586054714  $\pm$  1.147581420868i  &  5.478376425740  &   5.478376425759  \\ 
$\frac{\sqrt{2}}{256}$ &  133128 &  4.271669110480  $\pm$  1.147470349207i  &  5.476675105658  &   5.476675105674  \\ 
\bottomrule
\end{tabular}%
% }
\end{table}

\begin{table}[!htbp]
\centering
\caption{The first four eigenvalues obtained by IGA with $p=3$, where $\Omega = (0,1)^2$ and $n = 4$.}
\label{tab:k-case1-p3}
\begin{tabular}{ccccc}
\toprule
$h$ & DOFs & $k_{1,h}\,\, \&\,\, k_{2,h}$  & $k_{3,h}$ & $k_{4,h}$ \\
\midrule
$\frac{\sqrt{2}}{16}$ &  722 &  4.271571871821  $\pm$  1.147502581117i &    5.477217347103  &   5.477217347103  \\ 
$\frac{\sqrt{2}}{32}$ &  2450 &  4.271689020820  $\pm$  1.147437598159i    &  5.476174314713  &   5.476174314713  \\ 
$\frac{\sqrt{2}}{64}$ &  8978 &  4.271696372127  $\pm$  1.147433641270i   &  5.476112644596  &   5.476112644596  \\ 
$\frac{\sqrt{2}}{128}$ &  34322 &  4.271696833273  $\pm$  1.147433391759i   &  5.476108841177  &   5.476108841277  \\ 
$\frac{\sqrt{2}}{256}$ &  134162 &  4.271696861988  $\pm$  1.147433375371i   &  5.476108603885  &   5.476108604012  \\ 
\bottomrule
\end{tabular}%
\end{table}

\begin{figure}[H]%[!htbp]
\subfigure[$p = 2$]{\includegraphics[width=0.48\textwidth ]{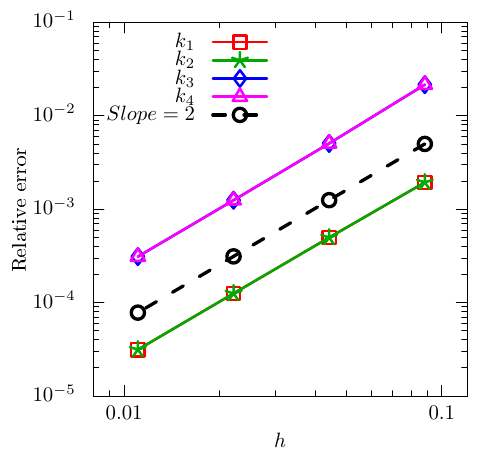}}
\subfigure[$p=3$]{\includegraphics[width=0.48\textwidth]{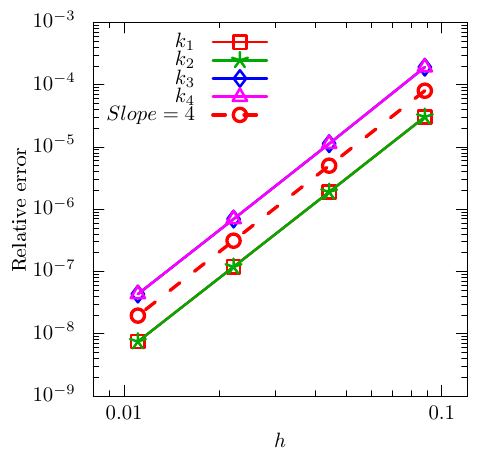}}
\caption{(Example 1) The convergence of the relative error for the first four transmission eigenvalues obtained by IGA with $p=2,3$.  } 
\label{fig:example_1_n_4}
  % \vspace{0.2in}
\end{figure}

\subsection{Example 2: Square domain with variable index of refraction}

We now consider the test case where  $\Omega = (0,1)^2$ and the index of refraction  $n= 8 + x -y$ \cite{xi2020high,HAN201796,ji2014multigrid}. We report the first four transmission eigenvalues obtained by IGA on six uniformly refined meshes with degrees $p=2$ and $p=3$ in Tables \ref{tab:k-case2-p2} and \ref{tab:k-case2-p3}, respectively. Similar to Example 1, it is observed that the computed real eigenvalues are monotonically decreasing as the mesh size $h$ becomes smaller. 

For comparison, in Table \ref{tab:k-case2-p3}, we also list the numerical results reported in \cite{xi2020high}, where a cubic $H^2$  nonconforming finite element scheme was considered. Note that the theoretical convergence order of the error for eigenvalues is 4 for both IGA with $p=3$ and the cubic nonconforming FEM \cite{xi2020high}. The comparison shows that to achieve the same level of accuracy, IGA requires only approximately one-sixth of the degrees of freedom (DOFs) needed by the cubic $H^2$ nonconforming FEM, demonstrating the significant advantage of IGA over the nonconforming FEM  \cite{xi2020high} for the transmission eigenvalue problem.
The convergence of the relative error of the first four transmission eigenvalues with respect to  $h$ is shown in Figure~\ref{fig:example_2_n_8_x_y}, which again confirms that the convergence order of the error for eigenvalues is $2(p-1)$ for $p\ge 2$.  

\begin{table}[!htbp]
\centering
\caption{The first four eigenvalues obtained by IGA with $p=2$,  where $\Omega = (0,1)^2$ and $n = 8+x-y$.}
\label{tab:k-case2-p2}
\begin{tabular}{cccccc}
\toprule
$h$ & DOFs & $k_{1,h}$ & $k_{2,h}$ & $k_{3,h}$ & $k_{4,h}$ \\
\midrule
$\frac{\sqrt{2}}{4}$ &  72 &  3.195370   & 4.271906    &  4.280797  &   4.582237  \\ 
$\frac{\sqrt{2}}{8}$ &  200 &  2.914875   & 3.742378    &  3.742628  &   4.325270  \\ 
$\frac{\sqrt{2}}{16}$ &  648 &  2.844952   & 3.588230    &  3.588502  &   4.167715  \\ 
$\frac{\sqrt{2}}{32}$ &  2312 &  2.827852   & 3.550979    &  3.551268  &   4.130110  \\ 
$\frac{\sqrt{2}}{64}$ &  8712 &  2.823603   & 3.541761    &  3.542054  &   4.120826  \\ 
$\frac{\sqrt{2}}{128}$ &  33800 &  2.822543   & 3.539462    &  3.539757  &   4.118513  \\ 
\bottomrule
\end{tabular}%
\end{table}

\begin{table}[!htbp]
\centering
\caption{The first four eigenvalues obtained by IGA with $p=3$,  where $\Omega = (0,1)^2$ and $n = 8+x-y$.}
\label{tab:k-case2-p3}
\begin{tabular}{cccccc}
\toprule
$h$ & DOFs & $k_{1,h}$ & $k_{2,h}$ & $k_{3,h}$ & $k_{4,h}$ \\
\midrule
$\frac{\sqrt{2}}{4}$ &  98 &  2.853625   & 3.657961    &  3.658273  &   4.206656  \\ 
$\frac{\sqrt{2}}{8}$ &  242 &  2.823687   & 3.543948    &  3.544239  &   4.122086  \\ 
$\frac{\sqrt{2}}{16}$ &  722 &  2.822275   & 3.538972    &  3.539267  &   4.117981  \\ 
$\frac{\sqrt{2}}{32}$ &  2450 &  2.822195   & 3.538713    &  3.539008  &   4.117756  \\ 
$\frac{\sqrt{2}}{64}$ &  8978 &  2.822190   & 3.538698    &  3.538993  &   4.117743  \\ 
$\frac{\sqrt{2}}{128}$ &  \textbf{34322} &  \textbf{2.822189}   & \textbf{3.538697}    &  \textbf{3.538992}  &   \textbf{4.117742}  \\ 
\midrule
\midrule
\cite{xi2020high} &  \textbf{194566} & \textbf{2.822189} & \textbf{3.538697} & \textbf{3.538992} & \textbf{4.117742} \\
\bottomrule
\end{tabular}%
\end{table}

\begin{figure}[!htbp]
\subfigure[$p = 2$]{\includegraphics[width=0.48\textwidth ]{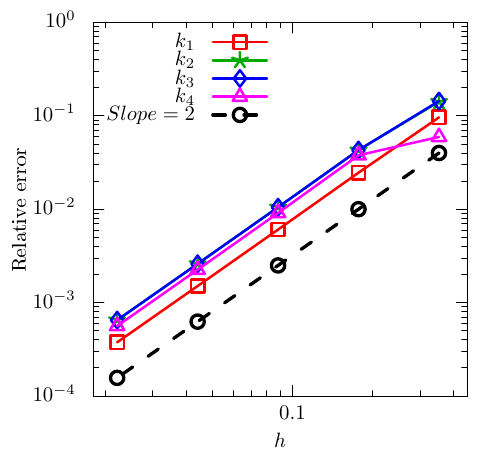}}
\subfigure[$p=3$]{\includegraphics[width=0.48\textwidth]{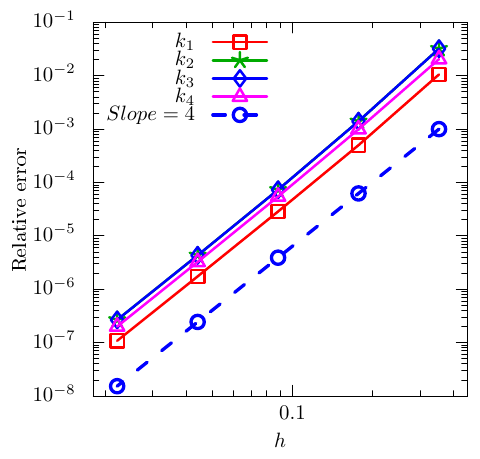}}
\caption{(Example 2) The convergence of the relative error for the first four transmission eigenvalues obtained by IGA with $p=2,3$.   } 
\label{fig:example_2_n_8_x_y}
  % \vspace{0.2in}
\end{figure}
% \FloatBarrier

\subsection{Example 3: Circular domain with variable index of refraction}

In this test case, we consider the disk-shaped domain $\Omega = \{ \bm x \in \mathbb{R}^2: |\bm x| < \displaystyle 1/2 \}$, and we show the initial mesh in Figure~\ref{fig:Disk geometry}. Following \cite{yang2016mixed}, we set the  index of refraction as $ n(\bm x) = 8 + 4|\bm x|$. We present the seven lowest  transmission eigenvalues obtained by IGA with degrees $p=2$ and $p=3$ in Tables \ref{tab:k-Example_3-p2} and \ref{tab:k-Example_3-p3}, respectively, where the mesh size $h$ is defined in \eqref{eq:mesh_size_h}. The results match well with those achieved in \cite{yang2016mixed}.
% Similar to the previous examples, the computed real eigenvalues are observed to decrease monotonically as $h$ decreases. 

\begin{figure}[!htbp]
    \centering
    \includegraphics[width=0.45\linewidth]{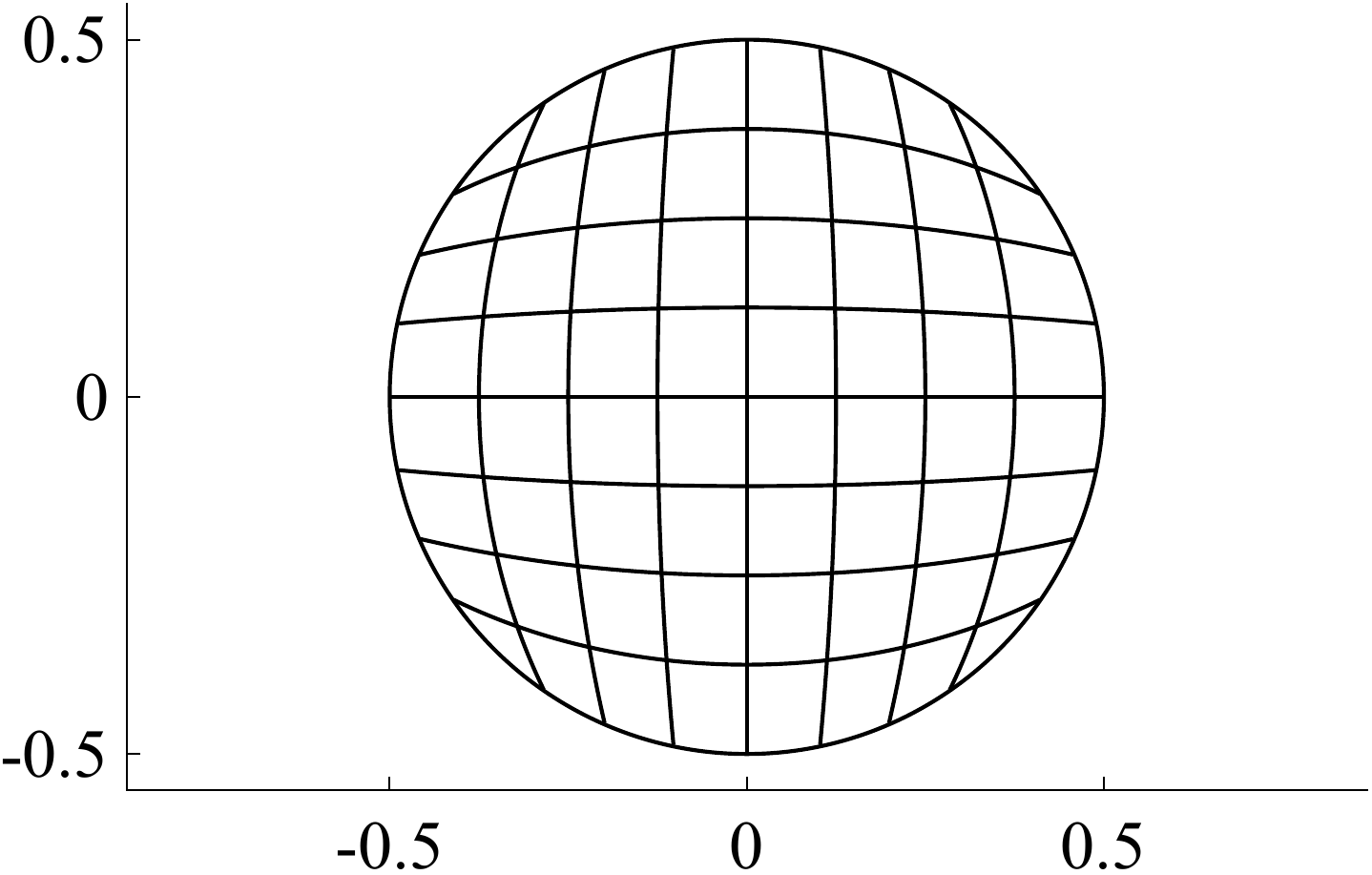}
    \caption{The initial physical mesh for a disk-shaped domain, consisting of $8\times 8$ elements.}
    \label{fig:Disk geometry}
\end{figure}

For comparison, in Table \ref{tab:k-Example_3-p3}, we also list the numerical results computed by the  mixed FEM using $\mathbb{P}_3$ element
 \cite{yang2016mixed}. Note that the theoretical convergence order of the error for eigenvalues is 4 for both IGA with $p=3$ and mixed FEM with the $\mathbb{P}_3$ element \cite{yang2016mixed}. The comparison shows that to achieve the same accuracy, the number of DOFs needed by IGA is around one hundredth of that used by the mixed FEM \cite{yang2016mixed}, demonstrating the significant advantage of IGA for the transmission eigenvalue problem in the curved domain.
The convergence of the relative error of the first  four different transmission eigenvalues ($k_i$, $i=1,2,4,6$) with respect to $h$ is displayed in Figure~\ref{fig:example_3_Disk}, which verifies the theoretical result.

\begin{table}[!htbp]
\centering
\caption{The transmission eigenvalues obtained by IGA with $p=2$,  where $\Omega = \{\bm x \in \mathbb{R}^2: |\bm x| < 1/2\} $ and $n(\bm x)= 8+4|\bm x|$.}
\label{tab:k-Example_3-p2}
\begin{tabular}{cccccc}
\toprule
$h$ & DOFs & $k_{1,h}$ & $k_{2,h}\approx k_{3,h}$ & $k_{4,h}\approx k_{5,h}$ & $k_{6,7,h}$ \\
\midrule
0.17677653 &  200 &  2.8318830   & 3.6979880    &  4.4941612  &   4.9153594  $\pm$  0.7880379i\\ 
0.08838833 &  648 &  2.7773889   & 3.5688046    &  4.3538115  &   4.8071479  $\pm$  0.8056438i\\ 
0.04419417 &  2312 &  2.7639086   & 3.5375744    &  4.3193552  &   4.7859802 $\pm$ 0.7999450i\\ 
0.02209709 &  8712 &  2.7605525   & 3.5298453    &  4.3108131  &   4.7812128 $\pm$ 0.7979725i\\ 
0.01104854 &  33800 &  2.7597144   & 3.5279181    &  4.3086824  &   4.7800569 $\pm$ 0.7974469i\\ 
\bottomrule
\end{tabular}%
\end{table}

\begin{table}[!htbp]
\centering
\caption{The transmission eigenvalues obtained by IGA with $p=3$,  where $\Omega = \{\bm x \in \mathbb{R}^2: |\bm x| < 1/2\} $ and $n(\bm x)= 8+4|\bm x|$.}
\label{tab:k-Example_3-p3}
\begin{tabular}{cccccc}
\toprule
$h$ & DOFs & $k_{1,h}$ & $k_{2,h}\approx k_{3,h}$ & $k_{4,h}\approx k_{5,h}$  & $k_{6,7,h}$ \\
\midrule
0.17677642 &  242 &  2.7608429   & 3.5324539    &  4.3160473  &   4.7847204 $\pm$ 0.8087511i\\ 
0.08838831 &  722 &  2.7595139   & 3.5275418    &  4.3083710  &   4.7798696 $\pm$ 0.7977539i\\ 
0.04419417 &  2450 &  2.7594399   & 3.5272919    &  4.3079959  &   4.7796854 $\pm$ 0.7972967i\\ 
0.02209709 &  \textbf{8978} &  \textbf{2.7594354} & \textbf{3.5272771}  &  \textbf{4.3079741}  &   \textbf{4.7796755 $\pm$ 0.7972705i} \\ 
0.01104854 &  34322 &  2.7594352   & 3.5272762    &  4.3079727  &   4.7796749 $\pm$ 0.7972689i \\ 
\midrule
\midrule
\cite{yang2016mixed} & \textbf{869487} & \textbf{2.7594400} & \textbf{3.5272823} & \textbf{4.3079799} & \textbf{4.7796829 $\pm$ 0.7972703i} \\
\bottomrule
\end{tabular}%
\end{table}

\begin{figure}[!htbp]
\subfigure[$p = 2$]{\includegraphics[width=0.48\textwidth ]{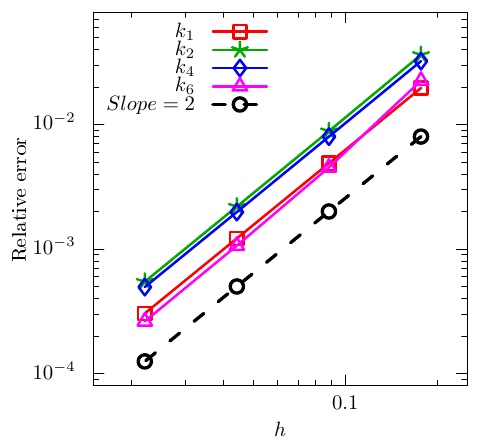}}
\subfigure[$p=3$]{\includegraphics[width=0.48\textwidth]{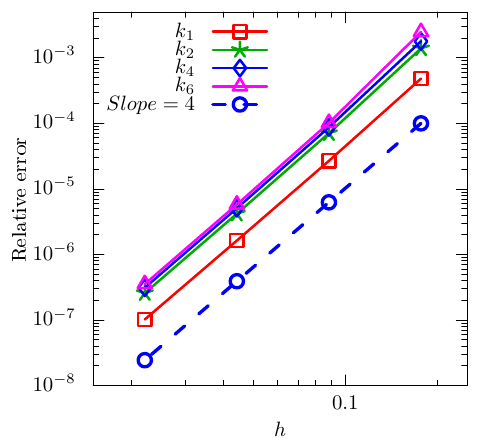}}
\caption{(Example 3) The convergence of the relative error of $k_i$ for $i=1,2,4,6$ with respect to $h$.  } 
\label{fig:example_3_Disk}
  % \vspace{0.2in}
\end{figure}

\FloatBarrier

\subsection{Example 4: L-shaped domain with constant index of refraction}

In this test case, following \cite{HAN201796,YANG2020112697}, we consider the L-shaped domain $\Omega_L := (-1,1)^2\backslash \big([0,1]\times [-1,0]\big)$, and show the initial mesh in Figure~\ref{fig:L-shaped domain geometry}. We set the index of refraction to $n=16$. The first four transmission eigenvalues achieved by IGA with $p=3$ on the finest mesh are 
\[
k_{1,h}\approx 1.4765675141, \,\, k_{2,h}\approx  1.5697294654,\,\, k_{3,h}\approx   1.7051723233, \mbox{ and} \,\, k_{4,h}\approx   1.7831163324,
\]
which match well with those reported in \cite{HAN201796,YANG2020112697}. Since the L-shaped domain has a re-entrant corner, the eigenfunctions may have low regularity. Consequently, the convergence order of the eigenvalue approximation may be suboptimal due to the geometric singularity, as displayed in Figure~\ref{fig:Example_4_L_shape_n_16}, where IGA with degree $p=2$ is used to obtain the results. Note that the solution to the biharmonic equation in the L-shaped domain $\Omega_L$ is in $H^{2+0.544}(\Omega_L)$, see \cite{Banz2017} and the references therein, which explains why the convergence rate for the error in the first transmission eigenvalue is only around 1.

\begin{figure}[h!]
    \centering
    \includegraphics[width=0.45\linewidth]{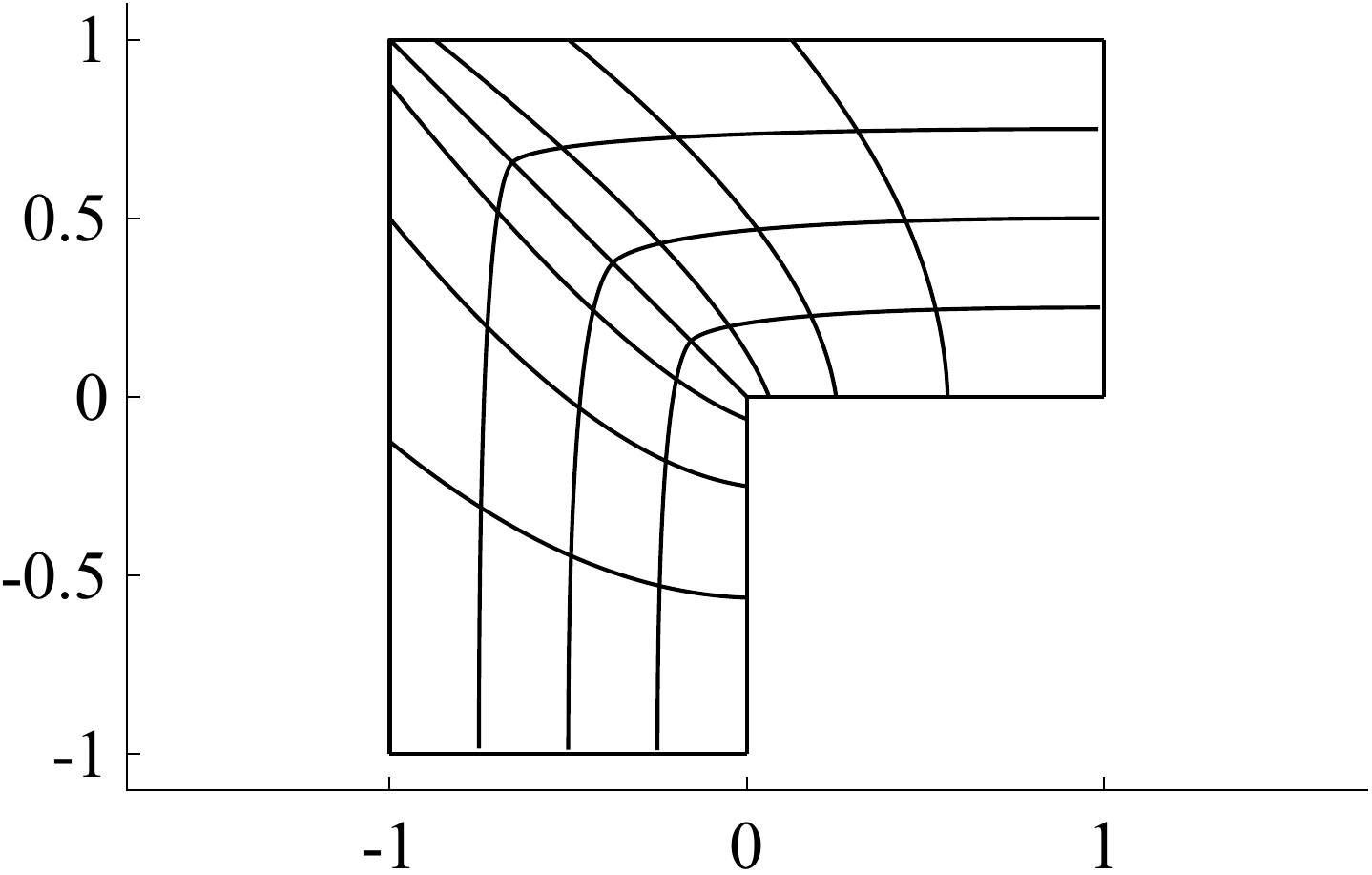}
    \caption{The initial physical mesh for the L-shaped domain, consisting of $8\times 4$ elements.}
    \label{fig:L-shaped domain geometry}
\end{figure}

\begin{figure}[h!]
\centering
 {\includegraphics[width=0.5\textwidth]{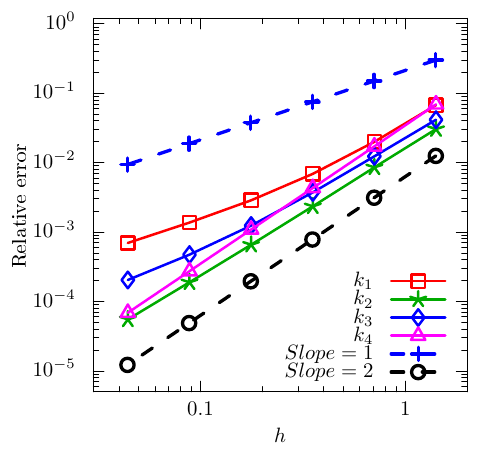}}
\caption{(Example 4) The convergence of the relative error of $k_i$ for $1\le i \le 4$ with respect to $h$ by IGA with $p=2$.  } 
\label{fig:Example_4_L_shape_n_16}
  % \vspace{0.2in}
\end{figure}

\FloatBarrier
\subsection{Example 5: Cube domain with  constant index of refraction}

In this example, we consider the cube domain  $\displaystyle \Omega = (- \frac{1}{2},\frac{1}{2})^3 $ and  constant index of refraction  $ n = 16$ \cite{yang2016mixed}. We present the four different lowest transmission eigenvalues achieved by IGA with degrees $p=2$ and $p=3$ in Tables \ref{tab:k-Example_4-p2} and \ref{tab:k-Example_4-p3}, respectively.  In Table \ref{tab:k-Example_4-p3}, we also show the results achieved by a mixed FEM using $\mathbb{P}_3$ element \cite{yang2016mixed}. Note that the theoretical convergence order of the error for eigenvalues is 4 for both IGA with $p=3$ and the mixed FEM with  $\mathbb{P}_3$ element. The comparison shows that to achieve the same accuracy, IGA with  $p=3$ requires approximately one-third of the DOFs needed by the mixed FEM with $\mathbb{P}_3$ element, demonstrating the advantage of IGA for the transmission eigenvalue problem.
The convergence of the relative error of the  transmission eigenvalues $k_i$ ($i=1,2,5, 8$) with respect to  $h$ is shown in Figure~\ref{fig:example_4_n_16}, which again confirms that the numerical convergence rate agrees well with the theoretical  results.

\begin{table}[!htbp]
\centering
\caption{The transmission eigenvalues obtained by IGA with $p=2$,  where $\Omega = (- \frac{1}{2},\frac{1}{2})^3 $ and $n=16$.}
\label{tab:k-Example_4-p2}
\begin{tabular}{cccccc}
\toprule
$h$ & DOFs & $k_{1,h}$ & $k_{2,h}\approx k_{3,h} \approx k_{4,h}$ & $k_{5,h}\approx k_{6,h} \approx k_{7,h}$ & $k_{8,h}$ \\
\midrule
$\frac{\sqrt{3}}{4}$  &  432    &  2.2430166   & 3.0275743    &  3.5443142  &   3.9767014  \\ 
$\frac{\sqrt{3}}{8}$  &  2000   &  2.1091323   & 2.6913663    &  3.1089618  &   3.4630555  \\ 
$\frac{\sqrt{3}}{16}$ &  11664  &  2.0775851   & 2.6106829    &  3.0162720  &   3.3154250  \\ 
$\frac{\sqrt{3}}{32}$ &  78608  &  2.0698098   & 2.5912625    &  2.9942761  &   3.2635522  \\ 
$\frac{\sqrt{3}}{64}$ &  574992 &  2.0678728   & 2.5864550    &  2.9888468  &   3.2508004  \\ 
\bottomrule
\end{tabular}%
\end{table}

\begin{table}[!htbp]
\centering
\caption{The transmission eigenvalues obtained by IGA with $p=3$,  where $\Omega = (- \frac{1}{2},\frac{1}{2})^3 $ and $n=16$.}
\label{tab:k-Example_4-p3}
\begin{tabular}{cccccc}
\toprule
$h$ & DOFs & $k_{1,h}$ & $k_{2,h}\approx k_{3,h} \approx k_{4,h}$ & $k_{5,h}\approx k_{6,h} \approx k_{7,h}$ & $k_{8,h}$ \\
\midrule
$\frac{\sqrt{3}}{4}$  &  686 &  2.0741996   & 2.6284321    &  3.0261084  &   3.2719371  \\ 
$\frac{\sqrt{3}}{8}$  &  2662 &  2.0676856   & 2.5869767    &  2.9890746  &   3.2595879  \\ 
$\frac{\sqrt{3}}{16}$ &  13718 &  2.0672561   & 2.5849735    &  2.9871645  &   3.2471818  \\ 
$\frac{\sqrt{3}}{32}$ &  \textbf{85750} &  \textbf{2.0672294}   & \textbf{2.5848637}    &  \textbf{2.9870505}  &   \textbf{3.2466048}  \\ 
$\frac{\sqrt{3}}{64}$ &  601526 &  2.0672278   & 2.5848572    &  2.9870436  &   3.2465719  \\ 
\midrule
\midrule
\cite{yang2016mixed}  &  \textbf{218453} &  \textbf{2.0672329}&  \textbf{2.5848674} &  \textbf{2.9870655} &  \textbf{3.2465923}\\
\bottomrule
\end{tabular}%
\end{table}

\begin{figure}[h!]
\subfigure[$p = 2$]{\includegraphics[width=0.48\textwidth ]{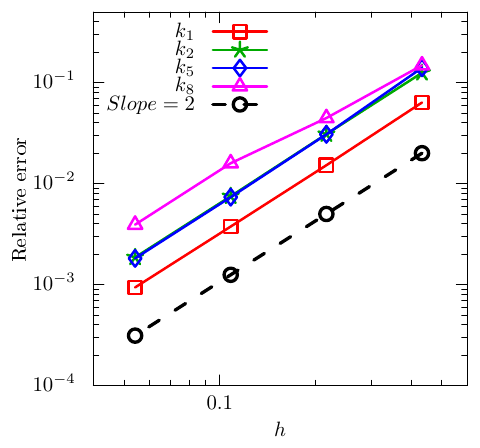}}
\subfigure[$p=3$]{\includegraphics[width=0.48\textwidth]{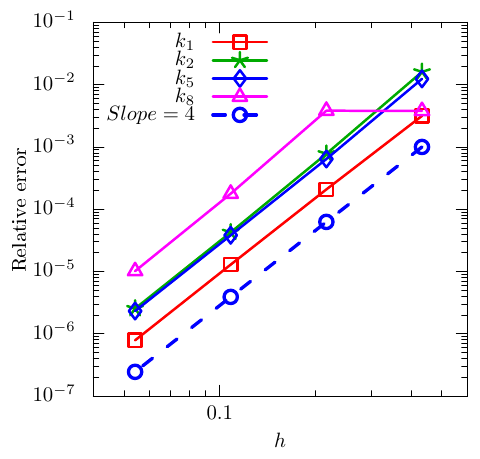}}
\caption{(Example 5) The convergence of the relative error of $k_i$ for $i=1,2,5,8$ with respect to  $h$.  } 
\label{fig:example_4_n_16}
  % \vspace{0.2in}
\end{figure}

\subsection{Example 6: A quarter of a hollow cylinder with variable index of refraction}

In our last numerical example, the index of refraction is set to be $n(\bm x) = 8 + |\bm x|$, for $\bm x\in \Omega_c\subset \mathbb{R}^3$, where $\Omega_c$ is a quarter of a hollow cylinder (a \textit{non-convex} domain) with inner radius $r = 1/2$ and outer radius $R=1$, which is given by 
\[
\Omega_c = \{\bm x = (x,y,z)\in \mathbb{R}^3: r^2 <  x^2 + y^2 < R^2,~  0 < x,y < R,   ~ 0 < z < 1 \}.
\]
The initial mesh is shown in Figure~\ref{fig:quarter_3D}.
We present the four lowest transmission eigenvalues achieved by IGA with degrees $p=2$ and $p=3$ in Tables \ref{tab:k-Example_6-p2} and \ref{tab:k-Example_6-p3}, respectively. 
%Similar to the previous examples, it is observed that he computed real eigenvalues are monotonically decreasing as  $h$ decreases. 
The convergence of the relative errors in the first four transmission eigenvalues with respect to $h$ is presented in Figure~\ref{fig:example_6_n_8_abs_x}, which again confirms our theoretical results.  

\begin{figure}[h!]
    \centering
\includegraphics[width=0.25\linewidth]{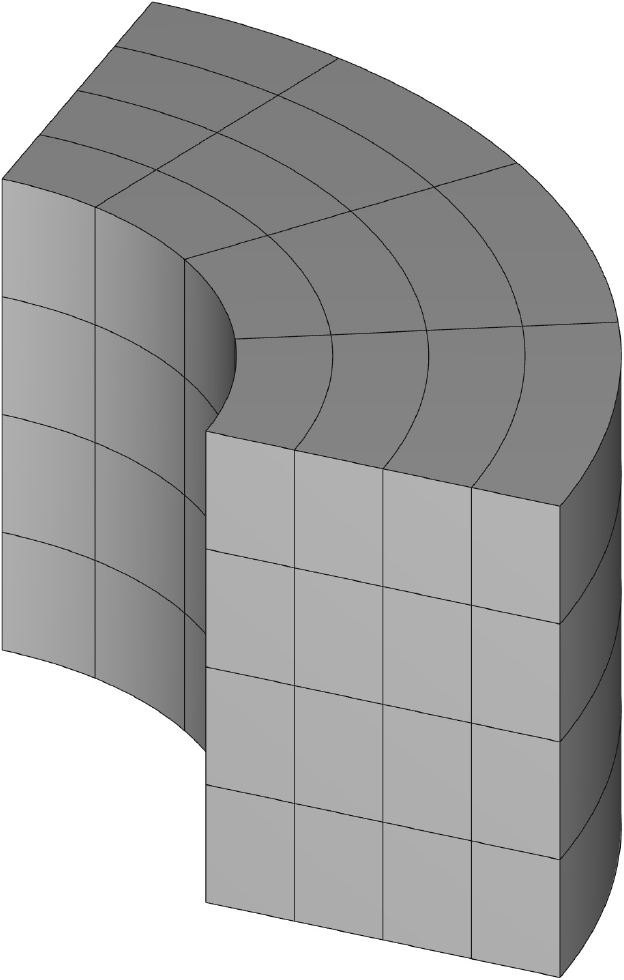}
\caption{The initial physical mesh for a quarter of a hollow cylinder, consisting of $4\times 4 \times 4$ elements.}
\label{fig:quarter_3D}
\end{figure}

\begin{table}[!htbp]
\centering
\caption{The transmission eigenvalues obtained by IGA with $p=2$,  where $\Omega$ is  a quarter of a hollow cylinder, and $n = 8+|\bm x|$.}
\label{tab:k-Example_6-p2}
\begin{tabular}{cccccc}
\toprule
$h$ & DOFs & $k_{1,h}$ & $k_{2,h}$ & $k_{3,h}$ & $k_{4,h}$ \\
\midrule
0.86501356 &  432 &  5.0627015774   & 5.1656498086    &  5.2164994319  &   5.4563280452  \\ 
0.43276146 &  2000 &  4.6691092793   & 4.6828864592    &  4.7786045727  &   4.9360538233  \\ 
0.21644376 &  11664 &  4.5634713523   & 4.5710408617    &  4.6658009247  &   4.8212704642  \\ 
0.10823755 &  78608 &  4.5370806175   & 4.5438308964    &  4.6381796484  &   4.7934633734  \\ 
0.05412269 &  574992 &  4.5304922785   & 4.5370766167    &  4.6313154356  &   4.7865649830  \\ 
\bottomrule
\end{tabular}%
\end{table}

\begin{table}[htbp]
\centering
\caption{The transmission eigenvalues obtained by IGA with $p=3$,  where $\Omega$ is a quarter of a hollow cylinder, and $n=8+|\bm x|$.}
\label{tab:k-Example_6-p3}
\begin{tabular}{cccccc}
\toprule
$h$ & DOFs & $k_{1,h}$ & $k_{2,h}$ & $k_{3,h}$ & $k_{4,h}$ \\
\midrule
0.86501356 &  686 &  4.5542922587   & 4.5938175494    &  4.6884799378  &   4.8236724817  \\ 
0.43276146 &  2662 &  4.5324679808   & 4.5373202838    &  4.6322289391  &   4.7876631340  \\ 
0.21644376 &  13718 &  4.5285371154   & 4.5349713066    &  4.6292236382  &   4.7844959679  \\ 
0.10823755 &  85750 &  4.5283116933   & 4.5348380361    &  4.6290428348  &   4.7842844932  \\ 
0.05412269 &  601526 &  4.5282980419   & 4.5348299897    &  4.6290318376  &   4.7842710928  \\ 
\bottomrule
\end{tabular}%
\end{table}

\begin{figure}[h!]
\subfigure[$p = 2$]{\includegraphics[width=0.48\textwidth ]{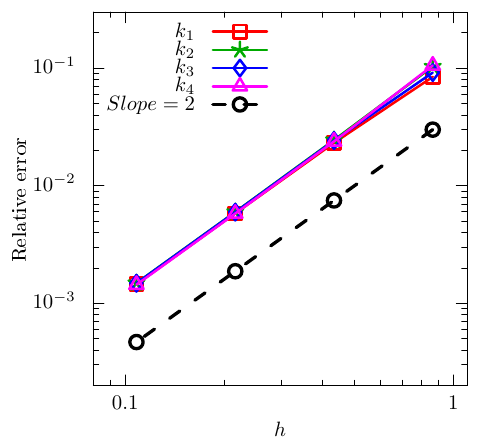}}
\subfigure[$p=3$]{\includegraphics[width=0.48\textwidth]{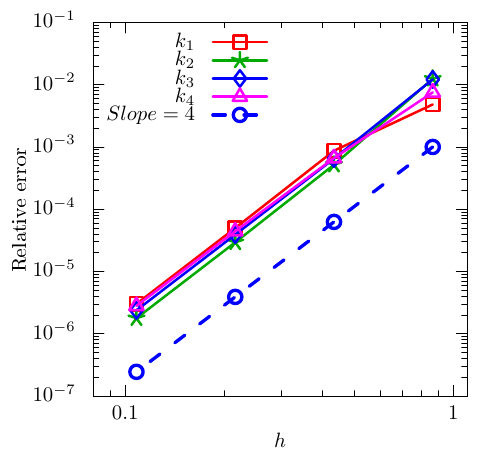}}
\caption{(Example 6) The convergence of relative errors in the first four transmission eigenvalues by IGA with $p=2,3$. } 
\label{fig:example_6_n_8_abs_x}
  % \vspace{0.2in}
\end{figure}

\section{Conclusion}
\label{sec-conclusion} 
\setcounter{equation}{0}
\setcounter{figure}{0}

In this paper, we introduce and rigorously analyze a $H^2$-conforming  geometry independent field approximation (GIFT) scheme, which is a generalization of the classical NURBS-based isogeometric analysis (IGA), for solving the fourth-order non-self-adjoint eigenvalue problem arising from Helmholtz transmission eigenvalue problem in  both two- and three-dimensional curved domains. By reformulating the quadratic fourth-order non-self-adjoint eigenvalue problem into an equivalent linear variational formulation defined on $\mathbb{H}=H_0^2(\Omega)\times L^2(\Omega)$, we establish the theoretical framework for the optimal error estimates of the transmission eigenvalues and eigenfunctions.

The main ingredients for our proof include (1)  the B-spline approximation theory in the physical domain, (2) the spectral approximation theory for the compact non-self-adjoint operators,  and (3) the Banach-space interpolation theory. Our theoretical results show that the eigenvalues and eigenfunctions converge with  $O(h^{2\sigma})$ and  $O(h^\sigma)$, respectively, where $\sigma\in(0,1]$ is the regularity parameter. Under higher regularity assumption, improved convergence rates for eigenvalues and eigenfunctions can be established, and they converge with $O(h^{2(p-1)})$ and  $O(h^{p-1})$, respectively. A variety of two- and three-dimensional numerical examples are provided to  confirm the theoretical results, and to show the advantages of the proposed scheme over the existing numerical methods for the transmission eigenvalue problem in terms of both numerical accuracy and geometric flexibility.

% Compared to existing methods such as mixed FEM and nonconforming FEM, our approach demonstrates significant advantages: exact geometry representation via NURBS, straightforward $H^2$-conforming discretization using smooth B-splines, and substantial reductions in degrees of freedom for comparable accuracy. These results establish IGA as a powerful and geometrically flexible tool for transmission eigenvalue problems.

% Future work includes extension to anisotropic media, development of adaptive refinement strategies for singular geometries, and integration with inverse scattering algorithms for practical parameter identification.

 % \section*{Acknowledgement}
 % The work of X. Meng was partially supported by the National Natural
 % Science Foundation of China (No. 12101057), Guangdong Higher Education Upgrading Plan (UIC-R0400024-21), and Guangdong and Hong Kong Universities “1+1+1” Joint Research Collaboration Scheme (No. 2025A0505000014).
 % The work of G. Hu was supported by the Science and Technology Development Fund, Macao SAR
 % (No. 0068/2024/RIA1, 001/2024/SKL), National Natural Science Foundation of China (No. 11922120), and MYRG of University of Macau (No. MYRG-GRG2025-00247-FST).

\FloatBarrier
\bibliographystyle{plain}
\bibliography{ref}

\end{document}